%% file: main_nilima.tex
\documentclass[]{elsarticle}
\usepackage{lineno,hyperref,fullpage,tikz-cd}

\usepackage{graphicx}
\usepackage{bm}
\usepackage{multirow}
\usepackage{amsmath}
\usepackage{amsfonts}
\usepackage{amssymb}
\usepackage{amsthm}
\usepackage{secdot}
\usepackage{accents}

\usepackage{multicol}
\usepackage{booktabs}
\usepackage{lscape}

\journal{Computers \& Mathematics with Applications}

\theoremstyle{plain}% Theorem-like structures provided by amsthm.sty
\newtheorem{theorem}{Theorem}[section]
\newtheorem{lemma}[theorem]{Lemma}

\theoremstyle{definition}

\newtheorem{example}[theorem]{Example}

\theoremstyle{remark}

\numberwithin{equation}{section}
\numberwithin{theorem}{section}
\numberwithin{remark}{section}

\catcode`\@=11
\newdimen\cdsep
\def\cdstrut{\vrule height .6\cdsep width 0pt depth .4\cdsep}
\def\@cdstrut{{\advance\cdsep by 2em\cdstrut}}

\def\arrow#1#2{
  \ifx d#1
    \llap{$\scriptstyle#2$}\left\downarrow\cdstrut\right.\@cdstrut\fi
  \ifx u#1
    \llap{$\scriptstyle#2$}\left\uparrow\cdstrut\right.\@cdstrut\fi
  \ifx r#1
    \mathop{\hbox to \cdsep{\rightarrowfill}}\limits^{#2}\fi
  \ifx l#1
    \mathop{\hbox to \cdsep{\leftarrowfill}}\limits^{#2}\fi
}
\catcode`\@=12

\usepackage{tikz}
\usetikzlibrary{positioning}

\newcommand{\K}{\Omega}  % we may change this, for the cubic pyramid
\newcommand{\C}[1]{\mathfrak{H}^#1} % for a cube
\newcommand{\Tr}[2]{\mathrm{Tr}[#1](#2)}  % trace operator 1
\newcommand{\tr}[2]{\mathrm{tr}[#1](#2)}  % trace operator 2
\newcommand{\Vkcirc}[2]{\accentset{\circ}{V}_{#1} \Lambda^#2} %these are k-th order shape functions on s-forms
\newcommand{\Vk}[2]{V_{#1} \Lambda^{#2}} %these are k-th order shape functions on s-forms
\newcommand{\Sk}[2]{\Sigma^{#1,#2}} % these are the dofs for the kth order shape functions for s-forms 

\newcommand{\VtoM}[1]{\mathcal{L}\left(#1\right)} % denoting a map from a 6-vector to a skew matrix.

\begin{document}

\begin{frontmatter}

\title{Conforming Finite Element Function Spaces in Four Dimensions, Part 1: Foundational Principles and the Tesseract}


\author{Nilima Nigam} 
\address{Department of Mathematics, Simon Fraser University, Burnaby, British Columbia BC V5C 2V3, Canada}

% \author{Gage S. Walters} 
% \address{The Johns Hopkins University Applied Physics Laboratory, Laurel, MD 20723, United States}

\author{David M. Williams \corref{mycorrespondingauthor}} 
\address{Department of Mechanical Engineering, The Pennsylvania State University, University Park, Pennsylvania 16802, United States}

\cortext[mycorrespondingauthor]{Corresponding author}
\ead{david.m.williams@psu.edu}

\begin{abstract}
The stability, robustness, accuracy, and efficiency of space-time finite element methods crucially depend on the choice of approximation spaces for test and trial functions. This is especially true for high-order, mixed finite element methods which often must satisfy an inf-sup condition in order to ensure stability. With this in mind, the primary objective of this paper and a companion paper is to provide a wide range of explicitly stated, conforming, finite element spaces in four-dimensions. In this paper, we construct explicit high-order conforming finite elements on 4-cubes (tesseracts); our construction uses tools from the recently developed `Finite Element Exterior Calculus'. With a focus on practical implementation, we provide details including Piola-type transformations, and explicit expressions for the volumetric, facet, face, edge, and vertex degrees of freedom. In addition, we establish important theoretical properties, such as the exactness of the finite element sequences, and the unisolvence of the degrees of freedom.
\end{abstract}


\end{frontmatter}

% \section{Introduction} \label{sec;introduction}

% \pagebreak
% \clearpage

\section{Introduction}

Our goal in this paper and the companion paper (part II) is to provide explicit high-order conforming families of finite element spaces on commonly used elements in $\mathbb{R}^4$. We consider three types of finite elements: a) hypercube elements which are generalizations of quadrilaterals to higher dimensional spaces, b) simplex elements which are generalizations of triangles to higher dimensional spaces, and c) hybrid elements which are tensor products of simplex elements with lower-dimensional simplex or hypercube elements. In accordance with principles of four-dimensional geometry, the \emph{tesseract} is an element of type a, the \emph{pentatope} is an element of type b, and the \emph{tetrahedral prism} is an element of type c; these latter elements are extensively used in space-time finite element methods which are increasingly important in science and engineering.

In part I of this paper, we discuss important principles of finite element construction, and develop conforming finite element spaces on the tesseract. Next, in part II, we develop conforming finite element spaces on the pentatope and tetrahedral prism. Finite element spaces on degenerate spatial elements such as the bipentatope and the cubic pyramid will be considered in future work.

We emphasize that finite element differential forms on tesseracts (and indeed $d$-dimensional cubes) have already been elegantly described within the Finite Element Exterior Calculus (FEEC) framework, see~\cite{arnold2014finite} for the analog of the $\mathcal{P}_k \Lambda^s$ spaces, and~\cite{arnold2015finite} for the analog of the $\mathcal{P}_k^{-} \Lambda^s$ spaces. The latter paper uses a tensor-product construction for the discrete de Rham complex; this is also our approach in this paper. The authors in~\cite{arnold2015finite} furthermore provide concrete approximation properties. 
%In more recent work \cite{gopalakrishnan2018auxiliary}, the authors specialize the FEEC to $\mathbb{R}^4$; the polynomial approximation spaces are built on simplicial meshes with the pentatope as a reference element.
The contribution of the present paper is to provide {\it fully explicit} finite element and bubble spaces for the particular de Rham complex described below on 4-cubical meshes; the degrees of freedom we prescribe are such that they are consistent (through traces) with well-known degrees of freedom on cubical meshes, see~\cite{monkbook}. Where possible, we work with proxies of the differential forms. Our aim is to aid practitioners as they implement these methods for four-dimensional problems.  

\subsection{Background}

We now turn our attention to a broader discussion of four-dimensional finite element methods. Generally speaking, we can construct a standard finite element method by: i) identifying a governing partial differential equation for the problem of interest, ii) determining the appropriate infinite-dimensional Sobolev spaces, iii) tessellating the domain into finite elements, iv) constructing suitable finite-dimensional subspaces of the infinite-dimensional Sobolev spaces on this domain, and v) developing weak formulations of the governing equations using the Galerkin approach. Naturally, this latter step involves using test functions and trial functions which are members of the aforementioned finite-dimensional subspaces, and judiciously employing integration by parts. This process is very well-understood for three-dimensional applications, and is not worthy of significant discussion here. However, in four dimensions, it is not immediately clear how to construct the infinite-dimensional Sobolev spaces, and the associated finite-dimensional subspaces. The main issue is that several new derivative operators arise in four dimensions, that do not have natural precedents in three dimensions. In addition, while the standard three-dimensional operators (such as the curl and divergence operators) can be extended into four dimensions, they usually have different domains and images. For example, the three-dimensional curl operator transforms three-vectors into three-vectors, whereas the four-dimensional curl operator transforms $4 \times 4$ skew-symmetric matrices into four-vectors. In this sense, the two curl operators have completely different characteristics. With this discussion in mind, some care is required as we construct suitable, conforming, finite elements in four-dimensional space.

Generally speaking, our approach for constructing finite element spaces in $ \mathbb{R}^4$ will be relatively unconcerned with spaces that are H1-conforming or L2-conforming, as these two spaces have been treated extensively elsewhere, (see the recent review paper of Frontin et al.~\cite{frontin2021foundations}). In fact, for these simpler cases, one may use Lagrange or Legendre basis functions for four-dimensional hypercube elements, and PKDO-type basis functions for four-dimensional simplex elements~\cite{proriol1957famille,koornwinder1975two,dubiner1991spectral,owens1998spectral}. Thereafter, conformity with H1 is enforced by allowing neighboring elements to share degrees of freedom that reside on the facets, faces, edges, and vertices. L2 conformity can be obtained by requiring that all element degrees of freedom are local, and are not shared across element interfaces. There are many examples of such methods in the literature; however, for now we will simply highlight the work of Diosady, Murman, and coworkers~\cite{diosady2015higher,diosady2017tensor,diosady2018linear,diosady2019scalable,franciolini2020multigrid}. They have spent considerable effort developing L2-conforming space-time finite element methods (space-time discontinuous Galerkin methods) on tesseracts, for simulating fluid dynamics and elasticity problems. These methods operate on partially-unstructured meshes formed by extruding three-dimensional unstructured meshes of hexahedra in the temporal direction in order to obtain meshes of tesseracts. The computational solution on each element is stabilized using `entropy variables'~\cite{hughes1986new}, in conjunction with space-time numerical fluxes that introduce dissipation which is proportional to jumps in the solution. A key advantage of these methods is that they leverage computational optimizations which are only possible for tensor-product elements.

As one might expect, additional technical challenges arise for spaces that reside between H1 and L2 in terms of smoothness, such as the H(curl) and H(div) spaces, (as we discussed above). These spaces reside in what we will henceforth refer to as the `H1-L2 gap'. Of course, naturally all H1-conforming spaces are also H(curl)- and H(div)-conforming, but we will only refer to a space as H(curl)-conforming if it fails to possess additional smoothness, i.e.~it does not simultaneously satisfy the smoothness requirements of H(curl) \emph{and} H1. Now, having established this terminology, we can state the main objective of this paper and its sequel more precisely: our objective is to identify explicit expressions for four-dimensional finite element spaces which are conforming to infinite-dimensional Sobolev spaces in the H1-L2 gap. These Sobolev spaces will be isotropic in nature, as they will maintain the same level of smoothness in all four coordinate directions.

It remains for us to discuss our intended approach for developing four-dimensional finite element spaces. Fortunately, this task can be accomplished in a relatively straightforward manner using techniques from FEEC. This framework was originally developed by Arnold, Falk, and Winther~\cite{arnold2006finite,arnold2010finite,arnold2018finite}, and has since been extended by many researchers. Broadly speaking, FEEC is a mathematical framework which uses the language of differential forms to construct conforming finite elements spaces in any number of dimensions. The resulting finite element spaces can be used to construct exact sequences and commuting diagrams, which facilitate proofs of stability and error estimates for the associated methods. In addition, FEEC can be used to rigorously classify different types of finite elements in accordance with the `Periodic Table of Finite Elements', developed by Arnold and Logg~\cite{arnold2014periodic}. Although FEEC is applicable to high-dimensional problems, to our knowledge, most of its explicit presentations have been limited to two or three dimensions. For example, Arnold et al.~\cite{arnold2008finite,arnold2014nonconforming} used FEEC to construct basis functions on three-dimensional tetrahedra for elasticity applications. Similar work has recently been undertaken by Chen and Huang~\cite{chen2022finite}. In addition, Licht~\cite{licht2022basis} has used FEEC to develop explicit basis functions in terms of barycentric coordinates for simplexes in two and three dimensions. Other efforts to simplify the implementation of FEEC-inspired simplex elements have been carried out by Kirby~\cite{kirby2014low} and Rognes et al.~\cite{rognes2010efficient}. Now, turning our attention to non-simplicial finite elements,  Arnold and Awanou~\cite{arnold2011serendipity} developed high-order basis functions on arbitrary-dimensional hypercubes, (although the explicit formulations were given on quadrilaterals and cubes). This work was followed up by additional papers on quadrilateral elements~\cite{arnold2005rectangular,arnold2015mixed}, and more generally on $d$-dimensional hypercubes~\cite{arnold2014finite,arnold2015finite}. Again, in the latter case, explicit representations are primarily provided for the two- and three-dimensional cases. Nigam and Phillips~\cite{nigam2012high,nigam2012numerical} used FEEC to construct high-order basis functions on three-dimensional square pyramids. Subsequently, Gillette extended this work to form Serendipity finite element spaces on square pyramids~\cite{gillette2016serendipity}. In addition, Natale~\cite{natale2017structure} and McRae et al.~\cite{mcrae2016automated} used FEEC to develop high-order basis functions on three-dimensional triangular prism elements. Thereafter, these elements were used to solve problems in the area of geophysics~\cite{natale2016compatible}. Finally, Gillette et al.~\cite{gillette2016construction} used FEEC and generalized barycentric coordinates to develop a set of basis functions on convex polygonal elements in two dimensions, and polyhedral elements in three dimensions.

One notable exception to the lower-dimensional efforts above, is the work of Golpalakrishnan et al.~\cite{gopalakrishnan2018auxiliary}. Broadly speaking, they used FEEC to construct auxiliary preconditioners for partial differential equations in four-dimensional space-time. More precisely, their work leveraged the following key observation: the solution to each differential equation can be decomposed into a highly-regular H1-conforming component, and a less-regular scalar or vector potential. This observation is clearly correct at the continuous level, but requires significant effort to fully realize at the discrete level. With this in mind, Gopalakrishnan et al.~used FEEC to construct four-dimensional Sobolev spaces, exact sequences, and discrete projection operators. The resulting machinery enabled the construction of a discrete analog of the continuous decomposition, and thereafter, the discrete decomposition was employed to construct auxiliary preconditioners. The preconditioners were successfully applied to low-order, simplex elements in $\mathbb{R}^{4}$. To the best of our knowledge, Gopalakrishnan et al.~were the first to explicitly articulate FEEC in four dimensions. A key motivation of our work is to expand that of Gopalakrishnan et al., and provide explicit, high-order basis functions on both simplicial and non-simplicial elements in four dimensions. We note that the de Rham sequence we consider is slightly different from that in \cite{gopalakrishnan2018auxiliary}; several examples involving PDEs (below) reveal why our particular choice is especially relevant for practical applications.

The main contributions of this work include:
\begin{enumerate}
\item A discussion of important scientific models motivating the need for a fully four-dimensional finite element theory. These additionally serve as motivation for the de Rham sequence under consideration (as mentioned previously).
\item An introduction to Sobolev spaces of $s$-forms, an overview of our preferred de Rahm sequence (itself), and a discussion of the link between isotropic and anisotropic Sobolev spaces.
\item A discussion of the traces associated with our particular de Rham sequence. To the best of our knowledge, this is the first time these traces have been discussed explicitly in the finite element literature for 1-, 2-, and 3-forms in $\mathbb{R}^4$. 

\item A review of tensorial high-order families of finite elements for tesseracts, and an explicit enumeration of basis functions.
\item A review of high-order families of finite elements for pentatopes, and an explicit enumeration of basis functions. This contribution is presented in part~II.
\item A description of high-order families of finite elements on tetrahedral prisms, and an explicit enumeration of basis functions. To the best of our knowledge, this is the first such construction in the literature. This contribution is presented in part~II.
\item An enumeration of easy-to-use degrees of freedom for each of the finite element spaces constructed above, which lead to globally conforming, unisolvent families. These (under assumptions of smoothness) satisfy the commuting diagram property.
\end{enumerate}

Before proceeding further, we note that FEEC's remarkable mathematical elegance and flexibility comes at the cost of using differential forms, which may not be intuitive for many scientists and engineers. For completeness and ease of exposition, we will provide a short introduction to four-dimensional differential forms in what follows. In addition, we will introduce `proxies' which serve as intuitive representations of these differential forms, via the language of linear algebra, (i.e.~vectors and matrices). Thereafter, we will provide two examples of practical applications which leverage these differential forms and their associated proxies. Lastly, we will provide an overview of the remainder of the paper.

\subsection{Differential Forms}

We begin by introducing the following generic set of differential forms
\begin{align*}
    \text{0-forms}, \qquad \omega \in \Lambda^{0}(\Omega), \qquad \omega &= \omega, \\[1.0ex]
    \text{1-forms}, \qquad \omega \in \Lambda^{1}(\Omega), \qquad \omega &= \omega_{1} dx^{1} + \omega_{2} dx^{2} + \omega_{3} dx^{3} + \omega_{4} dx^{4}, \\[1.0ex]
    \text{2-forms} , \qquad \omega \in \Lambda^{2}(\Omega), \qquad \omega &= \omega_{12} dx^{1} \wedge dx^{2} +\omega_{13} dx^{1} \wedge dx^{3} + \omega_{14} dx^{1} \wedge dx^{4} \\
    & + \omega_{23} dx^{2} \wedge dx^{3} + \omega_{24} dx^{2} \wedge dx^{4} + \omega_{34} dx^{3} \wedge dx^{4}, \\[1.0ex]
    \text{3-forms}, \qquad \omega \in \Lambda^{3}(\Omega), \qquad \omega &= \omega_{123} dx^{1} \wedge dx^{2} \wedge dx^{3} + \omega_{124} dx^{1} \wedge dx^{2} \wedge dx^{4} \\
    &+ \omega_{134} dx^{1} \wedge dx^{3} \wedge dx^{4} + \omega_{234} dx^{2} \wedge dx^{3} \wedge dx^{4}, \\[1.0ex]
    \text{4-forms}, \qquad \omega \in \Lambda^{4}(\Omega), \qquad \omega &= \omega_{1234} dx^{1} \wedge dx^{2} \wedge dx^{3} \wedge dx^{4}.
\end{align*}
Here, $\Lambda^{s}(\Omega)$ is the space of $s$-forms on the contractible region $\Omega$, where $s = 0, 1,2,3,4$. 

In a natural fashion, each differential form has a proxy which is obtained by applying a conversion operator (denoted by $\Upsilon_k$, $k = 0, 1, 2, 3, 4$) to each as follows
\begin{align*}
    \Upsilon_{0} \omega &= \omega, \qquad  \Upsilon_{4} \omega  = \omega_{1234}, \\[1.0ex]
    \Upsilon_{1} \omega &= \begin{bmatrix} 
    \omega_1 \\[1.0ex] \omega_2 \\[1.0ex] \omega_3 \\[1.0ex] \omega_4 \end{bmatrix}, \quad 
    \Upsilon_{2} \omega = \frac{1}{2} \begin{bmatrix}
    0 & \omega_{12} & \omega_{13} & \omega_{14} \\[1.0ex]
    -\omega_{12} & 0 & \omega_{23} & \omega_{24} \\[1.0ex]
    -\omega_{13} & -\omega_{23} & 0 & \omega_{34} \\[1.0ex]
    -\omega_{14} & -\omega_{24} & -\omega_{34} & 0
    \end{bmatrix}, \quad
    \Upsilon_{3} \omega = \begin{bmatrix} \omega_{234} \\[1.0ex] -\omega_{134} \\[1.0ex] \omega_{124} \\[1.0ex] -\omega_{123} \end{bmatrix}.
\end{align*}
Here, we observe that $\Upsilon_{2} \omega$ belongs to the space of $4\times4$ skew-symmetric matrices denoted by $\mathbb{K}$. These proxies were previously defined in~\cite{gopalakrishnan2018auxiliary}, with the exception of the 2-form proxy ($\Upsilon_{2} \omega$) which is defined for the first time above. 

\subsection{The Practical Significance of 1-forms}

The field of relativistic fluid dynamics provides a convenient example of the practical application of a $1$-form. Suppose that we consider a relativistic fluid equipped with a spatial velocity vector $u = \left(u_x, u_y, u_z\right)$. Then, we can define a 1-form that represents the 4-velocity
\begin{align*}
    \omega = \gamma \left( c dx^{1} + u_x dx^{2} + u_y dx^{3} + u_z dx^{4} \right),
\end{align*}
where $c$ is the speed of light, and $\gamma$ is the Lorentz factor
\begin{align*}
    \gamma = \frac{1}{\sqrt{1 - \frac{\left\| u \right\|^{2}}{c^2}} }.
\end{align*}
Here, we assume that the coordinates $\left(x_1,x_2,x_3,x_4\right)$ correspond to $\left(t,x,y,z\right)$, and that
\begin{align*}
    \left\| u \right\|^{2} \equiv u_{1}^{2} + u_{2}^{2} + u_{3}^{2}. 
\end{align*}
Furthermore, we set $U = \Upsilon_{1} \omega$ as the 1-form proxy of the 4-velocity, where
\begin{align*}
    \left\| U \right\|^{2} = \pm c^2 = g_{ij} U^{j} U^{i}.
\end{align*}
Here, $g_{ij}$ is a space-time metric tensor, and the choice of plus or minus (above) depends on the choice of metric tensor, (see section 14.5 of~\cite{kleppner2014introduction}). 
Next, we can construct $V$, the 2-form proxy of the 4-vorticity as follows
\begin{align*}
    V & = \Upsilon_{2} \psi = 2 \, \text{skwGrad} \left( \eta U \right),
\end{align*}
where skwGrad is a new, four-dimensional derivative operator that is defined in section~\ref{sobolev_sec}, $\eta$ is the chemical potential, and
\begin{align*}
    \psi &= V_{12} dx^{1} \wedge dx^{2} + V_{13} dx^{1} \wedge dx^{3} + V_{14} dx^{1} \wedge dx^{4} \\
    &+ V_{34} dx^{3} \wedge dx^{4} + V_{24} dx^{2} \wedge dx^{4} + V_{23} dx^{2} \wedge dx^{3},
\end{align*}
is the 2-form of the 4-vorticity. In addition, it can be shown that the following matrix-vector product vanishes for the isentropic case
\begin{align*}
     V U = 0.
\end{align*}
In this way, the 4-velocity acts as a zero eigenvector for the 4-vorticity. We refer the interested reader to~\cite{andersson2021relativistic}, section 5 for more details.

\subsection{The Practical Significance of 2-forms and 3-forms} \label{emag}

The field of electromagnetics provides a nice example of the practical significance of 2-forms and 3-forms. In particular, one may construct a 2-form which contains components of the electric field $E = \left(E_x, E_y, E_z\right)$ and the magnetic field $B = \left(B_x, B_y, B_z\right)$ as follows
\begin{align*}
    \omega &= -c( B_x dx^{1} \wedge dx^{2} + B_y dx^{1} \wedge dx^{3} + B_z dx^{1} \wedge dx^{4}) \\ 
    &- E_x dx^{3} \wedge dx^{4} + E_y dx^{2} \wedge dx^{4} - E_z dx^{2} \wedge dx^{3},
\end{align*}
where $c$ is the speed of light, and the generic coordinates $\left(x_1,x_2,x_3,x_4\right)$ correspond to the well-known temporal and spatial coordinates $\left(t,x,y,z\right)$. The 2-form above is called the Maxwell 2-form. We can also introduce the well-known Faraday 2-form
\begin{align*}
   \varphi &=B_x dx^{3} \wedge dx^{4} - B_y dx^{2} \wedge dx^{4} + B_z dx^{2} \wedge dx^{3} \\
    &-c( E_x dx^{1} \wedge dx^{2}  + E_y dx^{1} \wedge dx^{3} + E_z dx^{1} \wedge dx^{4} ),  
\end{align*}
see section 6.5 of~\cite{hubbard2015vector} for details. 

In addition, one may construct a 3-form which contains components of the electric current density $J = \left(j_x, j_y, j_z\right)$ as follows
\begin{align*}
    \sigma &= -\rho dx^{2} \wedge dx^{3} \wedge dx^{4} + j_x dx^{1} \wedge dx^{3} \wedge dx^{4} \\
    &-j_y dx^{1} \wedge dx^{2} \wedge dx^{4} + j_z dx^{1} \wedge dx^{2} \wedge dx^{3},
\end{align*}
where $\rho$ is the electric charge density. We can then define the associated proxies
\begin{align*}
    F &= \Upsilon_2 \omega, \quad G = \Upsilon_3 \sigma,\quad H = \Upsilon_2 \varphi. 
\end{align*}
Using these proxies, it turns out that the \emph{inhomogeneous Maxwell's equations} can be explicitly formulated as follows
\begin{align*}
    \text{curl} \left( F \right) = 4 \pi G, \qquad \text{curl} \left(H\right) = 0,
\end{align*}
where $\mathrm{curl}$ is the four-dimensional version of the curl operator, (cf.~section~\ref{sobolev_sec}). We note that the form proxies are given by
\begin{align*}
    F = \Upsilon_2 \omega &= \frac{1}{2}\begin{bmatrix}
    0 & -c B_x & -c B_y & -c B_z \\[1.0ex]
    c B_x & 0 & -E_z & E_y \\[1.0ex]
    c B_y & E_z & 0 & -E_x \\[1.0ex]
    c B_z & -E_y & E_x & 0
    \end{bmatrix}, \qquad G =\Upsilon_3 \sigma = -\begin{bmatrix} \rho \\[1.0ex] j_x \\[1.0ex] j_y \\[1.0ex] j_z \end{bmatrix}, \\[1.0ex]
    H =\Upsilon_2 \varphi &= \frac{1}{2} \begin{bmatrix}
        0 & -c E_x & -c E_y & -c E_z \\[1.0ex]
        c E_x & 0 & B_z & -B_y \\[1.0ex]
        c E_y & -B_z & 0 & B_x \\[1.0ex]
        c E_z & B_y & -B_x & 0
    \end{bmatrix}.
\end{align*}
We can also show that the following equation holds 
\begin{align*}
    \text{div} \left( G \right) = 0,
\end{align*}
where $\mathrm{div}$ is the four-dimensional version of the divergence operator, (cf.~section~\ref{sobolev_sec}). This latter equation enforces the conservation of the 4-current. For more details on the physical formulation above, please consult section 6.12 of~\cite{hubbard2015vector}.

\subsection{Overview of the Paper}

Our outline for the present paper is as follows. In section~2, we introduce a reference element, a generic mapping operator for the tesseract, and several guiding principles for constructing finite element spaces on the tesseract (and elsewhere). In section~3, we define four-dimensional Sobolev spaces and the associated Piola-type transformations (pullback operations) that enable mappings of scalars, vectors, and matrices between the reference element and physical elements. In section 4, we explicitly state the high-order finite element spaces for the  tesseract. In section 5, we prove that these spaces satisfy  interpolation and commuting diagram properties. Finally, in section~6, we summarize our presentation with some concluding remarks.

\input{guiding}

\section{Sobolev Spaces and Associated Mappings}\label{sobolev_sec}

In this section, we provide an explicit description of Sobolev spaces in four dimensions and their associated mapping operations. This discussion is not exhaustive, as an infinite number of such Sobolev spaces are possible to construct, but is merely intended to provide enough background to facilitate the development of conforming finite element spaces in subsequent sections. With this in mind, we can begin by constructing a set of first-derivative operators
\begin{align*}
    \text{grad}, \; \text{skwGrad}, \; \text{curl}, \; \text{div}.
\end{align*}
These derivative operators are sufficient for the purposes of defining a complete chain or cochain of operators. It is also useful to define the following `auxiliary' operators
\begin{align*}
    \text{Curl}, \; \text{Div},
\end{align*}
where `$\text{Curl}$' is similar in nature to `$\text{skwGrad}$', and `$\text{Div}$' is similar in nature to `$\text{curl}$'. The auxiliary operators are necessary for the formulation of trace operators and a dual chain complex.

In a natural fashion, the derivative operators above can be used to construct Sobolev spaces
\begin{align*}
    H\left( \text{grad}, \Omega, \mathbb{R} \right) &= \left\{u \in L^{2}\left(\Omega, \mathbb{R} \right) : \text{grad} \, u \in L^{2} \left(\Omega,\mathbb{R}^{4} \right) \right\} =:\mathcal{H}^0(\Omega), \\
    H\left( \text{skwGrad}, \Omega, \mathbb{R}^{4} \right) &= \left\{E \in L^{2}\left(\Omega, \mathbb{R}^{4} \right) : \text{skwGrad} \, E \in L^{2} \left(\Omega, \mathbb{K} \right) \right\} =:\mathcal{H}^1(\Omega), \\
    H\left( \text{curl}, \Omega, \mathbb{K} \right) &= \left\{F \in L^{2}\left(\Omega, \mathbb{K} \right) : \text{curl} \, F \in L^{2} \left(\Omega,\mathbb{R}^{4} \right) \right\}=:\mathcal{H}^2(\Omega), \\
    H\left( \text{div}, \Omega, \mathbb{R}^{4} \right) &= \left\{G \in L^{2}\left(\Omega, \mathbb{R}^{4} \right) : \text{div} \, G \in L^{2} \left(\Omega,\mathbb{R} \right) \right\}=:\mathcal{H}^3(\Omega),
\end{align*}
where we recall that $\mathbb{K}$ is a $4 \times 4$ skew-symmetric matrix, and
\begin{align*}
    L^{2}(\Omega,\mathbb{R}) &= \left\{q : \int_{\Omega} q^2 \, dx < \infty  \right\} =: \mathcal{H}^{4}(\Omega). 
\end{align*}
Furthermore 
\begin{align*}
    H\left( \text{Curl}, \Omega, \mathbb{R}^{4} \right) &= \left\{E \in L^{2}\left(\Omega, \mathbb{R}^{4} \right) : \text{Curl} \, E \in L^{2} \left(\Omega, \mathbb{K} \right) \right\}, \\
    H\left( \text{Div}, \Omega, \mathbb{K} \right) &= \left\{F \in L^{2}\left(\Omega, \mathbb{K} \right) : \text{Div} \, F \in L^{2} \left(\Omega,\mathbb{R}^{4} \right) \right\}.
\end{align*} 
The first five spaces above are associated with the standard \emph{de Rahm} complex in four dimensions
\begin{align}\label{eq:ourdehram}
	\begin{matrix}
		C^{\infty}(\Omega, \mathbb{R}) & \arrow{r}{\mathrm{grad}} & C^{\infty}(\Omega, \mathbb{R}^{4}) & \arrow{r}{\mathrm{skwGrad}} & C^{\infty}(\Omega, \mathbb{K}) & \arrow{r}{\mathrm{curl}} & C^{\infty}(\Omega, \mathbb{R}^{4})  & \arrow{r}{\mathrm{div}} & C^{\infty}(\Omega,\mathbb{R}).
	\end{matrix}
\end{align}
This complex is analogous to the standard de Rahm complex in \emph{three dimensions}
\begin{align*}
	\begin{matrix}
		C^{\infty}(\Omega, \mathbb{R}) & \arrow{r}{\nabla} & C^{\infty}(\Omega, \mathbb{R}^{3}) & \arrow{r}{\nabla \times} & C^{\infty}(\Omega, \mathbb{R}^{3}) & \arrow{r}{\nabla \cdot} & C^{\infty}(\Omega, \mathbb{R}),
	\end{matrix}
\end{align*}
where $\nabla$ is the three-dimensional gradient operator, $\nabla \times$ is the three-dimensional curl operator, and $\nabla\cdot$ is the three-dimensional divergence operator.

We can also form the L2 de Rahm complex 
\begin{align*}
	\begin{matrix}
		H(\mathrm{grad}, \Omega, \mathbb{R}) & \arrow{r}{\mathrm{grad}} & H(\mathrm{skwGrad}, \Omega, \mathbb{R}^{4}) & \arrow{r}{\mathrm{skwGrad}} & H(\mathrm{curl}, \Omega, \mathbb{K}) & \arrow{r}{\mathrm{curl}} & H(\mathrm{div},\Omega, \mathbb{R}^{4})  & \arrow{r}{\mathrm{div}} & L^{2}(\Omega,\mathbb{R}).
	\end{matrix}
\end{align*}
This complex is analogous to the L2 de Rahm complex in \emph{three dimensions}
\begin{align*}
	\begin{matrix}
		H(\nabla, \Omega, \mathbb{R}) & \arrow{r}{\nabla} & H(\nabla \times, \Omega, \mathbb{R}^{3}) & \arrow{r}{\nabla \times} & H(\nabla \cdot, \Omega, \mathbb{R}^{3}) & \arrow{r}{\nabla \cdot} & L^{2}(\Omega,\mathbb{R}).
	\end{matrix}
\end{align*}

In the formulas above, we let `$\text{grad}$' denote the four-dimensional gradient operator which can be applied to a scalar, $u \in L^{2}\left(\Omega, \mathbb{R} \right)$, such that $\left[\text{grad} u\right]_{i} = \partial_{i} u$ for $i = 1, \ldots, 4$. Next, we let  `$\text{skwGrad}$' denote an antisymmetric gradient operator which can be applied to a 4-vector, $E \in L^{2}\left(\Omega, \mathbb{R}^{4} \right)$, given by
\begin{align*}
    \left[ \text{skwGrad} \, E \right] = \frac{1}{2} \left( 
    \left[ \text{Grad} \, E\right]^{T} - \left[ \text{Grad} \, E\right] \right),
\end{align*}
where we define $\left[ \text{Grad} \, E\right]_{ij} = \partial_{j} E_{i}$, and
\begin{align*}
   \left[ \text{skwGrad} \, E \right] = \frac{1}{2}
   \begin{bmatrix}
        0 & \partial_{1} E_2 - \partial_{2} E_1 & \partial_{1} E_3 - \partial_{3} E_1 & \partial_{1} E_4 - \partial_{4} E_1 \\[1.0ex]
        \partial_{2} E_1 - \partial_{1} E_2 & 0 & \partial_{2} E_3 - \partial_{3} E_2 & \partial_{2} E_4 - \partial_{4} E_2 \\[1.0ex]
        \partial_{3} E_1 - \partial_{1} E_3 & \partial_{3} E_2 - \partial_{2} E_3 & 0 & \partial_{3} E_4 - \partial_{4} E_3 \\[1.0ex]
        \partial_{4} E_1 - \partial_{1} E_{4} & \partial_{4} E_2 - \partial_{2} E_4 & \partial_{4} E_3 - \partial_{3} E_4 & 0
   \end{bmatrix}.
\end{align*}
In addition, we let `$\text{curl}$' denote a derivative operator which can be applied to a $4 \times 4$ skew-symmetric matrix, $F \in L^{2}\left(\Omega, \mathbb{K} \right)$, given by
\begin{align*}
    \left[\text{curl} \, F \right]_{i} = \sum_{k,l=1}^{4} \varepsilon_{ijkl} \partial_{j} F_{kl},
\end{align*}
where $\varepsilon_{ijkl}$ is the Levi-Civita tensor, and
\begin{align*}
    \left[\text{curl} \, F \right] =
    \begin{bmatrix}
         \partial_{2} \left( F_{34} - F_{43} \right) + \partial_{3} \left(F_{42} - F_{24} \right) + \partial_{4} \left(F_{23} - F_{32} \right) \\[1.0ex]
         \partial_{1} \left(F_{43} - F_{34} \right) + \partial_{3} \left(F_{14} - F_{41} \right) + \partial_{4} \left(F_{31} - F_{13} \right) \\[1.0ex]
         \partial_{1} \left(F_{24} - F_{42} \right) + \partial_{2} \left(F_{41} - F_{14} \right) + \partial_{4} \left(F_{12} - F_{21} \right) \\[1.0ex]
         \partial_{1} \left(F_{32} - F_{23}\right) + \partial_{2} \left(F_{13} - F_{31}\right) + \partial_{3} \left(F_{21} - F_{12} \right)
    \end{bmatrix}.
\end{align*}
Here, the Levi-Civita tensor vanishes when its indices are repeated, and takes on the values of $1$ or $-1$ when the permutations of its indices are even or odd, respectively. Next, we let `$\text{div}$' denote the standard divergence operator which acts on a 4-vector, $G \in L^{2}\left(\Omega, \mathbb{R}^{4} \right)$, such that $\left[\text{div} \, G \right] = \partial_{i} G_{i}$. Furthermore, we let `$\text{Curl}$' denote the auxiliary curl operator which can be applied to a 4-vector, $E \in L^{2}(\Omega, \mathbb{R}^{4})$, such that
\begin{align*}
    \left[\text{Curl} \, E \right]_{ij} = \sum_{k,l=1}^{4} \varepsilon_{ijkl} \partial_{k} E_{l},
\end{align*}
and
\begin{align*}
    \left[\text{Curl} \, E \right] = 
    \begin{bmatrix}
        0 & \partial_{3} E_4 - \partial_{4} E_3 & \partial_{4} E_2 - \partial_{2} E_4 & \partial_{2} E_3 - \partial_{3} E_2 \\[1.0ex]
        \partial_{4} E_3 - \partial_{3} E_4 & 0 & \partial_{1} E_4 - \partial_{4} E_1 & \partial_{3} E_1 - \partial_{1} E_3 \\[1.0ex]
        \partial_{2} E_4 - \partial_{4} E_2 & \partial_{4} E_1 - \partial_{1} E_4 & 0 & \partial_{1} E_2 - \partial_{2} E_1 \\[1.0ex]
        \partial_{3} E_2 - \partial_{2} E_3 & \partial_{1} E_3 - \partial_{3} E_1 & \partial_{2} E_1 - \partial_{1} E_2 & 0
    \end{bmatrix}.
\end{align*}
Finally, we let `$\text{Div}$' denote the auxiliary divergence operator which can be applied to a $4 \times 4$ skew-symmetric matrix, $F \in L^{2}(\Omega, \mathbb{K})$, such that
\begin{align*}
    \left[\text{Div} \, F \right]_{i} = \sum_{j=1}^{4} \partial_{j} F_{ij},
\end{align*}
and
\begin{align*}
    \left[\text{Div} \, F \right] = 
    \begin{bmatrix}
        \partial_{2} F_{12} + \partial_{3} F_{13} + \partial_{4} F_{14} \\[1.0ex] 
        \partial_{1} F_{21} + \partial_{3} F_{23} + \partial_{4} F_{24} \\[1.0ex] 
        \partial_{1} F_{31} + \partial_{2} F_{32} + \partial_{4} F_{34} \\[1.0ex] 
        \partial_{1} F_{41} + \partial_{2} F_{42} + \partial_{3} F_{43}
    \end{bmatrix}.
\end{align*}
For completeness we record here the sequence used in \cite{gopalakrishnan2018auxiliary}: 
\begin{align*}
	\begin{matrix}
		C^{\infty}(\Omega, \mathbb{R}) & \arrow{r}{\mathrm{grad}} & C^{\infty}(\Omega, \mathbb{R}^{4}) & \arrow{r}{\mathrm{Curl}} & C^{\infty}(\Omega, \mathbb{K}) & \arrow{r}{\mathrm{Div}} & C^{\infty}(\Omega, \mathbb{R}^{4})  & \arrow{r}{\mathrm{div}} & C^{\infty}(\Omega,\mathbb{R}).
	\end{matrix}
\end{align*} The key difference between this and Eq.~\eqref{eq:ourdehram} is in the use of differentials acting on 1- and 2-forms. In Eq.~\eqref{eq:ourdehram} we see the operators which arise naturally in the PDE models presented in Section 1; this is an additional motivation for working with this sequence.

It can be shown that the first-derivative operators  satisfy the following identities
\begin{align}
    \Upsilon_{1} \left( d^{\left(0\right)} \omega \right) &= \text{grad} \left( \Upsilon_{0} \omega \right), \qquad \qquad \qquad \; \omega \in \Lambda^{0}(\Omega) := \mathcal{D}'(\Omega,\Lambda^{0}),\\
    \Upsilon_{2} \left( d^{\left(1\right)} \omega \right) &= \text{skwGrad} \left( \Upsilon_{1} \omega \right), \qquad \qquad \omega \in \Lambda^{1}(\Omega) := \mathcal{D}'(\Omega,\Lambda^{1}), \label{derivative_id_one} \\
    \Upsilon_{3} \left( d^{\left(2\right)} \omega \right) &= \text{curl} \left( \Upsilon_{2} \omega \right), \qquad \qquad \qquad \; \omega \in \Lambda^{2}(\Omega) := \mathcal{D}'(\Omega,\Lambda^{2}), \label{derivative_id_two} \\
    \Upsilon_{4} \left( d^{\left(3\right)} \omega \right) &= \text{div} \left( \Upsilon_{3} \omega \right), \qquad \qquad \qquad \; \, \omega \in \Lambda^{3}(\Omega) := \mathcal{D}'(\Omega,\Lambda^{3}).
\end{align}
The identities in Eqs.~\eqref{derivative_id_one} and \eqref{derivative_id_two} are proved in~\ref{derivative_appendix}. The remaining identities were proved previously in~\cite{gopalakrishnan2018auxiliary}. Based on these identities, it is immediately obvious that the following diagram commutes
\begin{align*}
\begin{matrix}
\mathcal{D}'(\Omega,\Lambda^{0}) & \arrow{r}{d^{\left(0\right)}} & \mathcal{D}'(\Omega,\Lambda^{1}) & \arrow{r}{d^{\left(1\right)}} & \mathcal{D}'(\Omega,\Lambda^{2}) & \arrow{r}{d^{\left(2\right)}} & \mathcal{D}'(\Omega,\Lambda^{3})  & \arrow{r}{d^{\left(3\right)}} & \mathcal{D}'(\Omega,\Lambda^{4})                 \cr
\arrow{d}{\Upsilon_0} &                      & \arrow{d}{\Upsilon_1} &   & \arrow{d}{\Upsilon_2} & & \arrow{d}{\Upsilon_3} & & \arrow{d}{\Upsilon_4} \cr
\mathcal{D}'(\Omega, \mathbb{R})                   & \arrow{r}{\text{grad}} & \mathcal{D}'(\Omega, \mathbb{R}^{4}) &  \arrow{r}{\text{skwGrad}} & \mathcal{D}'(\Omega, \mathbb{K}) &  \arrow{r}{\text{curl}} & \mathcal{D}'(\Omega, \mathbb{R}^{4})  &  \arrow{r}{\text{div}} & \mathcal{D}'(\Omega, \mathbb{R})            \cr
\end{matrix}
\end{align*}
It remains for us to characterize the behavior of our function spaces on the boundary of the domain, $\partial \Omega$. Towards this end, we can introduce the following trace identities for 1-forms
\begin{align}
    \nonumber \left(\text{tr}^{(1)} E \right)(F) &= \int_{\partial \Omega} \left( n \times E \right) : F \, ds \\[1.0ex]
     &= \int_{\Omega} \left(\text{Curl} \, E \right) : F \, dx - \int_{\Omega} \left(\text{curl} \, F \right) \cdot E \, dx,  \label{trace_one_A}
\end{align}
where $E \in H\left( \text{Curl}, \Omega, \mathbb{R}^{4} \right)$ and $F \in H\left( \text{curl}, \Omega, \mathbb{K} \right)$.
 %Similarly,
% %
% \begin{align}
%     \nonumber \left(\text{tr}^{(1)} E \right)(F) &= \int_{\partial \Omega} \left( n \times E \right) : F \, ds \\[1.0ex]
%      &= 4 \int_{\Omega} F \times (\mathrm{skwGrad} E) \, dx - \int_{\Omega} \left(\text{curl} \, F \right) \cdot E \, dx,  \label{trace_one_B}
% \end{align}
% %
% where $E \in H\left( \mathrm{skwGrad}, \Omega, \mathbb{R}^{4} \right)$ and $F \in H\left( \text{curl}, \Omega, \mathbb{K} \right)$. 
In a similar fashion,
\begin{align}
    \nonumber \left(\text{tr}^{(1)} E \right)(F) &= \frac{1}{2} \int_{\partial \Omega}  \left[  E \otimes n - n \otimes E\right] : F \, ds \\[1.0ex]
    &= \int_{\Omega} \left( \text{Div} \, F \right) \cdot E \, dx - \int_{\Omega} F : \left(\text{skwGrad} \, E \right) \, dx, \label{trace_one_C}
\end{align}
where $E \in H\left( \text{skwGrad}, \Omega, \mathbb{R}^{4} \right)$ and $F \in H\left( \text{Div}, \Omega, \mathbb{K} \right)$.
We can also introduce trace identities for 2-forms
\begin{align}
    \nonumber \left(\text{tr}^{(2)} F \right)(E) &= \int_{\partial \Omega} \left( n \times F \right) \cdot E \, ds \\[1.0ex]
     &= \int_{\Omega} \left(\text{curl} \, F \right) \cdot E \, dx -\int_{\Omega} \left(\text{Curl} \, E \right) : F \, dx,  \label{trace_two_A}
\end{align}
 where $E \in H\left( \text{Curl}, \Omega, \mathbb{R}^{4} \right)$ and $F \in H\left( \text{curl}, \Omega, \mathbb{K} \right)$.
 %Similarly
% %
% \begin{align}
%     \nonumber \left(\text{tr}^{(2)} F \right)(E) &= \int_{\partial \Omega} \left( n \times F \right) \cdot E \, ds \\[1.0ex]
%      &= \int_{\Omega} \left(\text{curl} \, F \right) \cdot E \, dx -4\int_{\Omega} F \times \left(\text{skwGrad} \, E \right) \, dx,  \label{trace_two_B}
% \end{align}
% %
% where $E \in H\left( \mathrm{skwGrad}, \Omega, \mathbb{R}^{4} \right)$ and $F \in H\left( \text{curl}, \Omega, \mathbb{K} \right)$.
% %
In addition, it can be shown that
\begin{align}
    \nonumber \left(\text{tr}^{(2)} \mathcal{M} \right)(\mathcal{E}) &= \int_{\partial \Omega} \mathcal{M} n \cdot \mathcal{E} \, ds \\[1.0ex]
    &= \int_{\Omega} \left( \text{Div} \, \mathcal{M} \right) \cdot \mathcal{E} \, dx - \int_{\Omega} \mathcal{M} : \left(\text{skwGrad} \, \mathcal{E} \right) \, dx, \label{trace_two_C}
\end{align}
where $\mathcal{E} \in H\left( \text{skwGrad}, \Omega, \mathbb{R}^{4} \right)$ and $\mathcal{M} \in H\left( \text{Div}, \Omega, \mathbb{K} \right)$,
or similarly
\begin{align}
    \nonumber \left(\text{tr}^{(2)} \mathcal{M} \right)(\mathcal{E}) &= \int_{\partial \Omega} \mathcal{M} n \cdot \mathcal{E} \, ds \\[1.0ex]
    &= \int_{\Omega} \left( \text{Div} \, \mathcal{M} \right) \cdot \mathcal{E} \, dx - \int_{\Omega} \mathcal{M} \times \left(\mathrm{Curl} \, \mathcal{E} \right) \, dx, \label{trace_two_D}
\end{align}
where $\mathcal{E} \in H\left( \text{Curl}, \Omega, \mathbb{R}^{4} \right)$ and $\mathcal{M} \in H\left( \text{Div}, \Omega, \mathbb{K} \right)$. Finally, for 3-forms we have that
\begin{align}
    \nonumber \left(\text{tr}^{(3)} G \right)(u) &= \int_{\partial \Omega} \left(G \cdot n \right) u \, ds \\[1.0ex]
    &= \int_{\Omega} \left(\text{div} \,G \right) u \, dx + \int_{\Omega} G \cdot \left( \text{grad} \, u \right) dx, \label{trace_three}
\end{align}
where $G \in H(\text{div}, \Omega, \mathbb{R}^{4})$ and $u \in H(\text{grad}, \Omega, \mathbb{R})$. In the discussion above, we note that the cross-product operator between 4-vectors is defined such that
\begin{align*}
    \left[M \times N \right]_{ij} &= \sum_{k,l=1}^{4} \varepsilon_{ijkl} M_k N_l, 
\end{align*}
and
\begin{align*}
    M\times N &=\begin{bmatrix}
        0 & M_{3} N_4 - M_{4} N_3 & M_{4} N_2 - M_{2} N_4 & M_{2} N_3 - M_{3} N_2 \\[1.0ex]
        M_{4} N_3 - M_{3} N_4 & 0 & M_{1} N_4 - M_{4} N_1 & M_{3} N_1 - M_{1} N_3 \\[1.0ex]
        M_{2} N_4 - M_{4} N_2 & M_{4} N_1 - M_{1} N_4 & 0 & M_{1} N_2 - M_{2} N_1 \\[1.0ex]
        M_{3} N_2 - M_{2} N_3 & M_{1} N_3 - M_{3} N_1 & M_{2} N_1 - M_{1} N_2 & 0
    \end{bmatrix},
\end{align*}
where $M \in \mathbb{R}^{4}$ and $N \in \mathbb{R}^{4}$.
In addition, the cross product operator between a 4-vector and a $4 \times 4$ skew-symmetric matrix is
\begin{align*}
    \left[M \times U\right]_{i} = \sum_{k,l=1}^{4} \varepsilon_{ijkl} M_{j} U_{kl},
\end{align*}
and
\begin{align*}
    M \times U  =  \begin{bmatrix}
         M_{2} \left( F_{34} - F_{43} \right) + M_{3} \left(F_{42} - F_{24} \right) + M_{4} \left(F_{23} - F_{32} \right) \\[1.0ex]
         M_{1} \left(F_{43} - F_{34} \right) + M_{3} \left(F_{14} - F_{41} \right) + M_{4} \left(F_{31} - F_{13} \right) \\[1.0ex]
         M_{1} \left(F_{24} - F_{42} \right) + M_{2} \left(F_{41} - F_{14} \right) + M_{4} \left(F_{12} - F_{21} \right) \\[1.0ex]
         M_{1} \left(F_{32} - F_{23}\right) + M_{2} \left(F_{13} - F_{31}\right) + M_{3} \left(F_{21} - F_{12} \right)
    \end{bmatrix},
\end{align*}
where $M \in \mathbb{R}^{4}$ and $U \in \mathbb{K}$. Finally, the cross-product operator between two, $4 \times 4$ skew-symmetric matrices is defined such that
\begin{align*}
    U \times V &= \sum_{1\leq i < j \leq 4} \; \sum_{1\leq k < l \leq 4} \varepsilon_{ijkl} U_{ij} V_{kl} \\[1.0ex]
    &= U_{12} V_{34} - U_{13} V_{24} + U_{14} V_{23} + U_{23} V_{14} - U_{24} V_{13} + U_{34} V_{12},
\end{align*}
where $U \in \mathbb{K}$ and $V \in \mathbb{K}$.

In the remainder of this paper, we will utilize the following trace definitions: 
\begin{align*}
    \text{0-forms} \qquad u &= \Upsilon_{0} \omega, \qquad \text{tr}(u) = u\vert_{\partial \Omega}, \\[1.0ex]
    \text{1-forms} \qquad E &= \Upsilon_{1} \omega, \qquad \text{tr}(E) = \frac{1}{2} \left(E \otimes n -  n \otimes E \right)\vert_{\partial \Omega}, \\[1.0ex]
    \text{2-forms} \qquad F &= \Upsilon_{2} \omega, \qquad \text{tr}(F) = \left(n \times F\right)\vert_{\partial \Omega}, \\[1.0ex]
    \text{3-forms} \qquad G &= \Upsilon_{3} \omega, \qquad \text{tr}(G) = \left(G\cdot n\right)\vert_{\partial \Omega},
\end{align*}
where
\begin{align*}
    &u \in H \left(\text{grad}, \Omega, \mathbb{R} \right), \qquad \qquad \; \; \, \text{tr}(u) \in H^{1/2}\left(\partial \Omega, \mathbb{R} \right), \\[1.0ex]
    &E \in H \left(\text{skwGrad}, \Omega, \mathbb{R}^{4} \right), \qquad \text{tr}(E) \in  H^{-1/2}\left(\partial \Omega, \mathbb{K} \right), \\[1.0ex]
   & F \in H \left(\text{curl}, \Omega, \mathbb{K} \right), \qquad \qquad \; \; \, \text{tr}(F) \in H^{-1/2}\left(\partial \Omega, \mathbb{R}^{4} \right), \\[1.0ex]
    &G \in H \left(\text{div}, \Omega, \mathbb{R}^{4} \right), \qquad \qquad \; \, \text{tr}(G) \in H^{-1/2}\left(\partial \Omega, \mathbb{R} \right).
\end{align*}
We omit traces on 4-forms because they are not well-defined. In addition, we note that detailed analysis is required to prove the inclusions of the traces within the appropriate fractional Sobolev spaces (above). For the sake of brevity, this topic will be explored in future work.

From the identities above, we can identify the traces for the $s$-forms on to any hyperplane. For example, suppose a simply connected Lipschitz domain $\Omega$ has a boundary with a non-trivial intersection with the $x_4=0$ hyperplane. Let  $\partial \Omega \cap \{x_4=0\} = \mathcal{F}$. On this hyperplane, the unit normal is $n=[0,0,0,1]^T$ and for a sufficiently smooth $s$-form:

\begin{itemize}
	\item If $s=0$ and $u = \Upsilon_{0} \omega$,
    \begin{equation}
        \tr{\mathcal F}{u} = u\vert_{\mathcal{F}} = u(x_1,x_2,x_3,0). 
    \end{equation}
    That is, the restriction of $u$ to $\mathcal{F}$. This trace can be identified with a scalar field $\Tr{\mathcal F}{u}$ which is a 0-form proxy on $\mathcal{F}$.
    \item If $s=1$ and $E = \Upsilon_{1} \omega$, 
    \begin{align}
        \nonumber \tr{\mathcal F}{E} &= \frac{1}{2} \left(E \otimes n - n \otimes E\right)\vert_{\mathcal{F}} \\[1.0ex] 
        \nonumber &= \frac{1}{2} \begin{bmatrix}
            0 & 0 & 0 & E_1(x_1,x_2,x_3,0) \\[1.0ex]
            0 & 0 & 0 & E_2(x_1,x_2,x_3,0) \\[1.0ex]
            0 & 0 & 0 & E_3(x_1,x_2,x_3,0) \\[1.0ex]
            -E_1(x_1,x_2,x_3,0) & -E_2(x_1,x_2,x_3,0) & -E_3(x_1,x_2,x_3,0) & 0
        \end{bmatrix} \\[1.0ex]
        &= \frac{1}{2} \VtoM{\left[0, 0, E_1(x_1,x_2,x_3,0), 0, E_2(x_1,x_2,x_3,0), E_3(x_1,x_2,x_3,0) \right]^{T}}.
        \label{trace_formula_one}
    \end{align}
    That is, the bivector trace of $E$ on to $\mathcal{F}$. This trace can be identified with a 3-vector $\Tr{\mathcal F}{E}$ which is a 1-form proxy on $\mathcal{F}$. 
    \item If $s=2$ and $F = \Upsilon_{2} \omega$, 
    \begin{equation}
        \tr{\mathcal F}{F} = \left(n \times F\right)\vert_{\mathcal{F}} = 2 \begin{bmatrix}
            F_{23}(x_1,x_2,x_3,0) \\[1.0ex]
            -F_{13}(x_1,x_2,x_3,0) \\[1.0ex]
            F_{12}(x_1,x_2,x_3,0) \\[1.0ex]
            0
        \end{bmatrix}.
    \end{equation}
    That is, the tangential trace of $F$ on to $\mathcal{F}$. This trace can be identified with a 3-vector $\Tr{\mathcal F}{F}$ which is a 2-form proxy on $\mathcal{F}$.
	\item If $s=3$ and $G = \Upsilon_{3} \omega$, 
    \begin{equation}
        \tr{\mathcal F}{G} = \left(G \cdot n\right)\vert_{\mathcal{F}} = G_{4}(x_1,x_2,x_3,0).
    \end{equation}
    That is, the normal trace of $G$ on to $\mathcal{F}$. This trace can be identified with a scalar field $\Tr{\mathcal F}{G}$ which is a 3-form proxy on $\mathcal{F}$.
    \item If $s = 4$, the trace is not well-defined.
\end{itemize}
These identifications will be used while constructing the {\it bubble spaces:} subspaces of finite element polynomial spaces whose traces vanish on the boundary of the element. 

Having establishing the Sobolev spaces and the corresponding derivative and trace identities, we introduce the pullback operator $\phi^{\ast}$ of the differential forms $\omega$
\begin{align}
    u &= \Upsilon_0 \omega, \quad \forall u \in H \left(\text{grad}, \Omega, \mathbb{R} \right), \qquad \qquad \Upsilon_0 \phi^{\ast} \omega = u \circ \phi, \\[1.0ex]
    E &= \Upsilon_{1} \omega, \quad \forall E \in H \left(\text{skwGrad}, \Omega, \mathbb{R}^{4} \right), \quad \;  \Upsilon_{1} \phi^{\ast} \omega = D \phi^{T} \left[E \circ \phi \right], \\[1.0ex] 
    F &= \Upsilon_{2} \omega, \quad \forall F \in H \left(\text{curl}, \Omega, \mathbb{K} \right), \qquad \qquad \; \Upsilon_{2} \phi^{\ast} \omega =  D\phi^{T} \left[F \circ \phi \right] D\phi, \label{new_map_orig} \\[1.0ex] 
    G &= \Upsilon_{3} \omega, \quad \forall G \in H \left(\text{div}, \Omega, \mathbb{R}^{4} \right), \quad \quad \qquad \Upsilon_{3} \phi^{\ast} \omega = \left| D\phi \right| D \phi^{-1} \left[ G \circ \phi \right], \label{piola_map_orig} \\[1.0ex]
    q &= \Upsilon_{4} \omega, \quad \forall q \in L^{2} \left(\Omega, \mathbb{R} \right), \qquad \qquad \qquad \; \; \Upsilon_{4} \phi^{\ast} \omega = \left| D \phi \right| \left[ q \circ \phi \right],
\end{align}
where $[D\phi]_{ij} = \partial_{j}\phi_i$ is the Jacobian matrix. We note that Eq.~\eqref{piola_map_orig} is the space-time Piola transformation, which was previously obtained in~\cite{bazilevs2008isogeometric,gopalakrishnan2017mapped}. In addition, the tensor transformation in Eq.~\eqref{new_map_orig} appears to be new. One may consult~\ref{pull_back_appendix} for a proof of this transformation.

Finally, consider the trace-free Sobolev spaces
\begin{align*}
    H_{0}\left( \text{grad}, \Omega, \mathbb{R} \right) &= \left\{u \in L^{2}\left(\Omega, \mathbb{R} \right) : \text{grad} \, u \in L^{2} \left(\Omega,\mathbb{R}^{4} \right), \, \text{tr}(u) = 0 \right\}, \\
    H_{0}\left( \text{skwGrad}, \Omega, \mathbb{R}^{4} \right) &= \left\{E \in L^{2}\left(\Omega, \mathbb{R}^{4} \right) : \text{skwGrad} \, E \in L^{2} \left(\Omega, \mathbb{K} \right), \, \text{tr}(E) = 0 \right\}, \\
    H_{0}\left( \text{curl}, \Omega, \mathbb{K} \right) &= \left\{F \in L^{2}\left(\Omega, \mathbb{K} \right) : \text{curl} \, F \in L^{2} \left(\Omega,\mathbb{R}^{4} \right), \, \text{tr}(F) = 0 \right\}, \\
    H_{0}\left( \text{div}, \Omega, \mathbb{R}^{4} \right) &= \left\{G \in L^{2}\left(\Omega, \mathbb{R}^{4} \right) : \text{div} \, G \in L^{2} \left(\Omega,\mathbb{R} \right), \, \text{tr}(G) = 0 \right\}, 
\end{align*}
and
\begin{align*}
    H_{0}\left( \text{Curl}, \Omega, \mathbb{R}^{4} \right) &= \left\{E \in L^{2}\left(\Omega, \mathbb{R}^{4} \right) : \text{Curl} \, E \in L^{2} \left(\Omega, \mathbb{K} \right), \, \text{tr}(E) = 0 \right\}, \\
    H_{0}\left( \text{Div}, \Omega, \mathbb{K} \right) &= \left\{F \in L^{2}\left(\Omega, \mathbb{K} \right) : \text{Div} \, F \in L^{2} \left(\Omega,\mathbb{R}^{4} \right), \, \text{tr}(F) = 0 \right\}.
\end{align*}
Using several of these spaces, we can construct the following dual chain complex
\begin{align*}
	\begin{matrix}
		L^{2}(\Omega,\mathbb{R}) & \arrow{l}{-\mathrm{div}} & H_{0}(\mathrm{div}, \Omega, \mathbb{R}^{4}) & \arrow{l}{\mathrm{Div}} & H_{0}(\mathrm{Div}, \Omega, \mathbb{K}) & \arrow{l}{\mathrm{Curl}} & H_{0}(\mathrm{Curl},\Omega, \mathbb{R}^{4})  & \arrow{l}{-\mathrm{grad}} & H_{0}\left( \text{grad}, \Omega, \mathbb{R} \right).
	\end{matrix}
\end{align*}
This particular de Rahm complex is used frequently throughout~\cite{gopalakrishnan2018auxiliary}.

\subsection{Relationship Between Isotropic and Anisotropic Sobolev Spaces}

Often, the space-time domain $\Omega$ can be expressed as the tensor product of a one-dimensional time interval $\left(0, T\right)$ and a spatial domain $\Omega_{x}$. In accordance with~\cite{langer2019space}, under these circumstances, we may now construct the following anisotropic Sobolev spaces
\begin{align*}
    H^{r,s}\left(\Omega \right) = L^{2}\left(0,T; H^{r}(\Omega_x) \right) \cap H^{s} \left(0,T; L^{2}(\Omega_x) \right),
\end{align*}
where evidently $r \geq 0$, $s \geq 0$, for a sufficiently smooth spatial domain $\Omega_x$. In addition, we can define the following norm
\begin{align*}
    \left\| u \right\|_{H^{r,s}(\Omega)}^{2} = \left\| u \right\|_{L^{2}(\Omega)}^{2} + \left| u \right|_{L^{2}\left(0,T; H^{r}(\Omega_x) \right)}^{2} + \left| u \right|_{H^{s} \left(0,T; L^{2}(\Omega_x) \right)}^{2},
\end{align*}
where 
\begin{align*}
    \left| u \right|_{L^{2}\left(0,T; H^{r}(\Omega_x) \right)}^{2} &= \int_{\Omega_x} \int_{\Omega_x} \frac{\left\| u(x, \cdot ) - u(y,\cdot)\right\|^{2}_{L^{2}(0,T)}}{\left| x-y \right|^{d+2r}} \, dy \, dx, \\[1.0ex] \\
    \left| u \right|_{H^{s}\left(0,T; L^{2}(\Omega_x) \right)}^{2} &= \int_{0}^{T} \int_{0}^{T} \frac{\left\| u(\cdot, t ) - u(\cdot,\tau)\right\|^{2}_{L^{2}(\Omega_x)}}{\left| t-\tau \right|^{1+2s}} \, d\tau \, dt.
\end{align*}
By inspection, we have that
\begin{align*}
    H^{s,s}(\Omega) = H^{s}(\Omega).
\end{align*}
In this fashion, the anisotropic and isotropic Sobolev spaces coincide when $r = s$. We can now introduce the following embedding theorem
\begin{theorem}
    Suppose that $r, s\ \in \left[0,1\right]$. Then, the following continuous embeddings hold
    \begin{align*}
        H^{\max(r,s)}(\Omega) \hookrightarrow H^{r,s}(\Omega) \hookrightarrow  H^{\min(r,s)}(\Omega).
    \end{align*}
\end{theorem}
\begin{proof}
    The proof follows immediately from the arguments in Lemma 2.1 of~\cite{langer2019space}. Note, it is necessary to substitute $\Omega$ in place of $\Sigma$ in the original theorem.
\end{proof}
Based on the theorem above, we have a clear relationship between isotropic and anisotropic Sobolev spaces. It follows that many space-time problems which are naturally associated with anisotropic Sobolev spaces can still be approximated using isotropic Sobolev spaces. In particular, we can use isotropic Sobolev spaces for finite element methods as long as they are a subspace of the anisotropic Sobolev space associated with our problem. For example, we merely require that
\begin{align*}
    H^{p}(\Omega) \subset H^{r,s}(\Omega),
\end{align*}
for some suitable choice of $p$.

In what follows, we present an alternative approach for avoiding anisotropic Sobolev spaces: namely, the use of isotropic Sobolev spaces in conjunction with Lagrange multipliers.

\begin{example} 
    Consider the following parabolic  problem with $f \in L^{2}(0,T; H^{-1}(\Omega_x))$ and $g \in H^{1}_{0}(\Omega_x)$
    \begin{align*}
        \frac{\partial u}{\partial t} - \Delta u &= f, \qquad \text{in} \quad \Omega \\
        u &= 0, \quad \; \, \text{on} \quad \partial \Omega_x \times (0,T) \\
        u &= g, \quad \; \; \text{on} \quad \Omega_x \times \{t = 0\}.
    \end{align*}
    We can construct the standard weak formulation
    \begin{align*}
        \int_{0}^{T} \int_{\Omega_x} \frac{\partial u}{\partial t} w \, dx \, dt + \int_{0}^{T} \int_{\Omega_x} \nabla u \cdot \nabla w 
        \, dx \, dt = \int_{0}^{T} \int_{\Omega_x} f w \, dx \, dt,
    \end{align*}
    where the test function $w$ and the weak solution $u$ satisfy the following 
    \begin{align*}
        u, w \in L^{2}\left(0,T; H^{1}_{0}(\Omega_x) \right) \cap H^{1} \left(0,T; H^{-1}(\Omega_x) \right),
    \end{align*}
    in accordance with Chapter 7 of Evans~\cite{evans2010partial}.
    We can also obtain a weak solution in a suitable subspace of this original space   
    \begin{align*}
        u, w \in L^{2}\left(0,T; H^{1}_{0}(\Omega_x) \right) \cap H^{1} \left(0,T; L^{2}(\Omega_x) \right) \subset L^{2}\left(0,T; H^{1}_{0}(\Omega_x) \right) \cap H^{1} \left(0,T; H^{-1}(\Omega_x) \right).
    \end{align*}
    In either case, the Sobolev space of interest is anisotropic. Fortunately, this issue can be ameliorated by constructing an alternative weak formulation which uses Lagrange multipliers $\lambda$
    \begin{align*}
        \int_{0}^{T} \int_{\Omega_x} \frac{\partial u}{\partial t} w \, dx \, dt + \int_{0}^{T} \int_{\Omega_x} \nabla u \cdot \nabla w \, dx \, dt + \int_{0}^{T} \int_{\partial \Omega_x} \lambda \, w \, ds \, dt &= \int_{0}^{T} \int_{\Omega_x} f w \, dx \, dt, \\
        \int_{0}^{T} \int_{\partial \Omega_x} u \, \mu \,ds \, dt &= 0,
    \end{align*}
    where 
    \begin{align*}
        u, w &\in H^{1}(\Omega) \equiv  L^{2}\left(0,T; H^{1}(\Omega_x) \right) \cap H^{1} \left(0,T; L^{2}(\Omega_x) \right),
    \end{align*}
    and
    \begin{align*}
        \lambda, \mu \in L^{2}(0, T; H^{-1/2}(\partial \Omega_x)).
    \end{align*}
    In this latter formulation, the solution and the test functions both reside in an \emph{isotropic} Sobolev space $H^{1}(\Omega)$. We remark that many problems which are ``naturally" associated with anisotropic Sobolev spaces can be reformulated and re-associated with isotropic Sobolev spaces in a similar fashion. 
\end{example}

\input{tesseract}

\section{Interpolation and the Commuting Diagram Property}

We begin by considering a generic reference element $\widehat{K}$. In principle, this element can be a tesseract, pentatope, tetrahedral prism, or similar element. The remainder of this section will focus on the reference tesseract, but a similar analysis of other reference elements is possible. 

In this paper, we have introduced families of high-order conforming and unisolvent finite element spaces $(\widehat{K},\Vk{k}{s}(\widehat{K}), \Sk{k}{s}(\widehat{K}))$ on the reference element $\widehat{K}=\C{4}$. It is easy to check that these properties are retained under affine mappings of the element. By construction, the polynomial finite element spaces are exact on $\widehat{K}$. It remains for us to discuss interpolation and the commuting diagram property. We define interpolation operators $\Pi_{\widehat{K}}^s$ for suitably smooth $s$-forms $v \in \mathcal{H}^s(\widehat{K})$ by requiring 
\begin{align*}
    \ell( p - \Pi_{\widehat{K}}^sp) =0 , \qquad \forall \ell \in \Sk{k}{s}(\widehat{K}),
\end{align*}
where the $\ell$'s are the linear functionals associated with our degrees of freedom. We note that the local interpolants $\Pi_K^s$ on to mapped elements $K$ in physical space,  are defined in a similar fashion.

\begin{theorem}
Let $\Omega$ be a Lipschitz polyhedral domain in $\mathbb{R}^4$, tessellated by affine-mapped copies of a single reference element $\widehat{K}$ where $\widehat{K} = \C{4}$.

For $s=0,1,2,3,4$, let
\begin{align*}
    \mathtt{V}_k^{s}(\Omega):=\left\{ u_h\in \mathcal{H}^s(\Omega) \, : \, u_h \vert_{K} \in \Vk{k}{s}(K), \quad \forall \, \mbox{mapped} \,  K \, \mbox{in the mesh} \right\}.
\end{align*}
Define the global interpolant operator $\pi^s_h$
 for sufficiently smooth $s$-forms, $p\in \mathcal{H}^s(\Omega)$, via the local interpolation operators $\Pi_{K}^s$ as
\begin{align*}
    (\pi_h^s p)\vert_K = \Pi^s_Kp, \qquad \forall \, \mbox{mapped} \, K \, \mbox{in the mesh}.
\end{align*}
Then:
\begin{itemize}
\item For s=0,1,2,3 
\begin{align}
d^{(s)}  \mathtt{V}_k^{s}(\Omega) \subset  \mathtt{V}_k^{s+1}(\Omega).
\label{eq:nesting}
\end{align}
\item Let $p\in \mathcal{H}^s(\Omega)$ be sufficiently smooth so that $\Pi^s_K(p) $ and $\pi^s_h(p)$ are well-defined. Then for $s=0,1,2,3$
\begin{equation}\label{eq:commute} d^{(s)} \pi_h^s p  = \pi_h^{s+1} d^{(s)} p.\end{equation}. 
\end{itemize} 
That is, with these global interpolants, the following commuting diagram property holds
\begin{align*}
\begin{matrix}
\mathcal{U}^{0}(\Omega) & \arrow{r}{d^{\left(0\right)}} & \mathcal{U}^{1}(\Omega) & \arrow{r}{d^{\left(1\right)}} & \mathcal{U}^{2}(\Omega) & \arrow{r}{d^{\left(2\right)}} & \mathcal{U}^{3}(\Omega)  & \arrow{r}{d^{\left(3\right)}} & \mathcal{U}^{4}(\Omega)                 \cr
\arrow{d}{\pi_{h}^{0}} &                      & \arrow{d}{\pi_{h}^{1}} &   & \arrow{d}{\pi_{h}^{2}} & & \arrow{d}{\pi_{h}^{3}} & & \arrow{d}{\pi_{h}^{4}} \cr
\mathtt{V}_k^{0}(\Omega)                   & \arrow{r}{d^{(0)}} & \mathtt{V}_k^{1}(\Omega) &  \arrow{r}{d^{(1)}} & \mathtt{V}_k^{2}(\Omega) &  \arrow{r}{d^{(2)}} & \mathtt{V}_k^{3}(\Omega)  &  \arrow{r}{d^{(3)}} & \mathtt{V}_k^{4}(\Omega)            \cr
\end{matrix}
\end{align*} 
where $\mathcal{U}^s(\Omega)$ is a subspace of $\mathcal{H}^s(\Omega)$ such that the degrees of freedom on it are well-defined.

\end{theorem}

\begin{proof}
We present the proof for meshes comprised of affine-mapped copies of $\widehat{K}=\C{4}$. 
    For the first statement (Eq.~\eqref{eq:nesting}), let $v_h \in \mathtt{V}_k^s(\Omega).$ Then $v_h \in \mathcal{H}^s(\Omega)$ by definition.  Proceeding element-by-element for each tesseract $K$ in the tessellation, it is easy to check that
    \begin{align*}
        d^{(s)}[v_h \vert_K] \in \Vk{k}{s+1}(K).
    \end{align*}
    This proves the first statement.
    
    To prove the next statement (Eq.~\eqref{eq:commute}), it is enough to show that the degrees of freedom of $d^{(s)} \pi_h^s p$ and  $\pi_h^{s+1} d^{(s)} p$ agree on each tesseract $K$, and indeed, it suffices to show the statement on the reference tesseract $\C{4}$. We proceed in steps, by establishing Eq.~\eqref{eq:commute} for $s=0,1,2,3$ in turn.

Let $s=0$ and let $e$ be an edge of $\widehat{K} = K=\C{4}$ with tangent $\tau$. In addition, let $A$ and $B$ denote the endpoints of the edge.  We can compute:
\begin{align*}
    \int_e \mathrm{Tr} \left(d^{(0)} \Pi^0_K p - \Pi_K^1 d^{(0)} p\right) \cdot \tau q \, ds = \int_e \mathrm{Tr} \left(d^{(0)} \Pi^0_K p - d^{(0)}p\right) \cdot \tau q \, ds,\quad \forall q \in P^{k}(e)
\end{align*}
which follows from the definition of the interpolant $\Pi_K^1$. Next, by applying integration-by-parts, we obtain
\begin{align*}
    \int_e \mathrm{Tr} \left(d^{(0)} \Pi^0_K p -  d^{(0)}p\right) \cdot \tau q \, ds  &= \int_e \left(\frac{\partial}{\partial s} \mathrm{Tr}(\Pi^0_K p - p) \right) \cdot \tau q \, ds \\[1.0ex]
    &= \tau q \cdot \mathrm{Tr}(\Pi^0_K p - p)\vert_A^B - \int_e \tau \cdot \left(\frac{\partial q}{\partial s} \right) \mathrm{Tr}(\Pi^0_K p - p) \, ds, \qquad  \forall q \in P^{k}(e).
\end{align*} 
But since $\Pi^{0}_{K}$ is a 0-form interpolant at the vertices, the first term on the RHS above must vanish. In addition, the second term must vanish in accordance with Eq.~\eqref{eq:edge-1}. This shows that the edge degrees of freedom of $d^{(0)} \Pi^0_K p$  and $\Pi_K^1 d^{(0)}p$ agree. A similar argument for the face, facet, and volume degrees of freedom gives the desired result (Eq.~\eqref{eq:commute}) for $s=0$.

Next, let $s=1$. In addition, let 
\begin{align*}
q \in \VtoM{\begin{bmatrix}
			Q^{k-1,k-1,k-2,k-2}\\[1.0ex]
			Q^{k-1,k-2,k-1,k-2}\\[1.0ex]
			Q^{k-1,k-2,k-2,k-1}\\[1.0ex]
			Q^{k-2,k-1,k-1,k-2}\\[1.0ex]
			Q^{k-2,k-1,k-2,k-1}\\[1.0ex]
			Q^{k-2,k-2,k-1,k-1}
			 \end{bmatrix}}, 
\end{align*}
and consider the \emph{volumetric} degrees of freedom
\begin{align*}
\int_{K}\left(d^{(1)} \Pi^1_K p -  \Pi^2_K d^{(1)} p\right):q \, dx &=\int_{K}\left(d^{(1)} \Pi^1_K p -  d^{(1)}p\right):q \, dx =\int_{K} \mathrm{skwGrad} \left( \Pi^1_K p -  p\right):q \, dx \\[1.0ex]
&=\int_K (\Pi^1_Kp-p)\cdot \mathrm{Div}(q) \, dx - \int_{\partial K}qn \cdot (\Pi^1_Kp-p) \, ds \\[1.0ex]
&=\int_K (\Pi^1_Kp-p)\cdot \mathrm{Div}(q) \, dx - \int_{\partial K} \mathrm{Tr}(qn) \cdot \mathrm{Tr}((\Pi^1_Kp-p)) \, ds,
\end{align*}
where the second-to-last line follows from Eq.~\eqref{trace_two_C}. On the RHS of the expression above, the integral over the volume $K$ vanishes by the definition of $\Pi_K^1$, the fact that $\mathrm{Div}(q) \in Q^{k-1,k-2,k-2,k-2} \times Q^{k-2,k-1,k-2,k-2} \times Q^{k-2,k-2,k-1,k-2} \times Q^{k-2,k-2,k-2,k-1} (K)$, and Eq.~\eqref{eq:tesseract1a}. In addition, it can be easily checked that for any of the normals on $K$, the quantity $qn$ can be identified with a vector $r \in Q^{k-1,k-2,k-2}\times Q^{k-2,k-1,k-2}\times Q^{k-2,k-2,k-1}(\partial K)$, and therefore the integral over $\partial K$ will also vanish in accordance with the definition of the interpolant $\Pi_K^1$, and Eq.~\eqref{eq:hex1}. This argument tells us that the volume dofs of $d^{(1)} \Pi^1_K p $ and $ \Pi^2_Kd^{(1)}p$ agree.

We now show that the \emph{facet} degrees of freedom of the 2-forms $d^{(1)} \Pi^1_K p $ and $ \Pi^2_Kd^{(1)}p$ agree. To this end, let $\mathcal{F}$ be a facet of $K$, and $q\in Q^{k-2,k-1,k-1}\times Q^{k-1,k-2,k-1}\times Q^{k-1,k-1,k-2}(\mathcal{F})$. We compute
\begin{align*}
    \int_{\mathcal{F}} \mathrm{Tr}\left(d^{(1)} \Pi^1_K p -  \Pi^2_Kd^{(1)}p\right)\cdot q \, dx &= \int_{\mathcal{F}} \mathrm{Tr}\left(d^{(1)} \Pi^1_K p -d^{(1)}p\right)\cdot q \, dx \\[1.0ex]
    &=\int_{\mathcal{F}} d^{(1)} \mathrm{Tr}\left(\Pi^1_K p -p\right)\cdot q \, dx = \int_{\mathcal{F}} \nabla \times \mathrm{Tr}\left(\Pi^1_K p -p\right)\cdot q \, dx\\[1.0ex]
    &= \int_{\mathcal{F}} \mathrm{Tr}\left(\Pi^1_K p -p\right) \cdot (\nabla \times q) \, dx - \int_{\partial \mathcal{F}} (\mathrm{Tr}\left(\Pi^1_K p -p\right)\times \nu) \cdot q \, ds.
\end{align*}
In the first term on the RHS above, $\nabla \times q \in Q^{k-1,k-2,k-2}\times Q^{k-2,k-1,k-2}\times Q^{k-2,k-2,k-1}(\mathcal{F})$. As a result, the integral over $\mathcal{F}$ vanishes in accordance with Eq.~\eqref{eq:hex1} and the definition of the interpolant. The integral over $\partial \mathcal{F}$ vanishes as well, in accordance with Eq.~\eqref{eq:square-1} and the definition of the interpolant.

Next, consider the \emph{face} degrees of freedom, and suppose that $f$ is a face of the tesseract. Let $q \in Q^{k-1,k-1}(f)$ and let $\nu$ be the outer normal, perpendicular to the face. We denote by $\nu_{f}$ the outer normal to the face $f$ in the plane containing $f$.
In accordance with the proof of Theorem 6.7 in~\cite{monkbook}:
\begin{align*}
    \int_f \mathrm{Tr}\left(d^{(1)} \Pi^1_K p -  \Pi^2_Kd^{(1)}p\right)\cdot \nu q \, dx&= \int_f \mathrm{Tr}\left(d^{(1)} \Pi^1_K p -d^{(1)}p\right)\cdot \nu q \, dx \\[1.0ex]
    &=\int_f d^{(1)} \mathrm{Tr}\left(\Pi^1_K p -p\right)\cdot \nu q \, dx = \int_f \nabla \times \mathrm{Tr}\left(\Pi^1_K p -p\right)\cdot \nu q \,dx \\[1.0ex]
    &=\int_f (\nabla \times \nu) \cdot \mathrm{Tr}\left(\Pi^1_K p -p\right)q \, dx -\int_f \nabla \cdot(\nu \times \mathrm{Tr}\left(\Pi^1_K p -p\right))q \, dx\\[1.0ex]
    &=-\int_f \nabla \cdot(\nu \times \mathrm{Tr}\left(\Pi^1_K p -p\right))q \, dx\\[1.0ex]
    &=\int_f (\nu \times \mathrm{Tr}\left(\Pi^1_K p -p\right))\cdot \nabla q \, dx- \int_{\partial f} \nu_{f} \cdot (\nu \times \mathrm{Tr} \left(\Pi^1_K p -p\right)) q \, ds.
\end{align*}
In the first term on the RHS above, $\nabla q \in Q^{k-2,k-1}\times Q^{k-1,k-2}(f).$ As a result, the integral over $f$ vanishes in accordance with Eq.~\eqref{eq:square-1} and the definition of the interpolant. The integral over $\partial f$ vanishes as well, in accordance with Eq.~\eqref{eq:edge-1} and the definition of the interpolant.
 
The edge degrees of freedom for $d^{(1)}\Pi_K^1p$ and $\Pi_K^2 d^{(1)}p$ can be shown to agree using very similar calculations, which are omitted here.

Now let $s=2$. In addition, let $q\in Q^{k-2,k-1,k-1,k-1}\times Q^{k-1,k-2,k-1,k-1}\times Q^{k-1,k-1,k-2,k-1}\times Q^{k-1,k-1,k-1,k-2} (K)$ for which the $\mathrm{Curl}$ operator is well-defined. We then consider an integral over the \emph{volumetric} degrees of freedom
\begin{align*}
    \int_{K} \left(d^{(2)} \Pi^2_K p -  \Pi^3_Kd^{(2)}p\right)\cdot q \, dx &= \int_{K} d^{(2)}\left(\Pi^2_K p - p\right)\cdot q \, dx = \int_{K} \text{curl}\left(\Pi^2_K p - p\right)\cdot q \, dx\\[1.0ex]
    &=\int_{K} \left(\Pi^2_K p - p\right): \text{Curl} (q) \, dx -\int_{\partial K} \mathrm{Tr}(n\times q) \cdot \mathrm{Tr}\left(\Pi^2_K p - p\right) \, ds,
\end{align*}
where the last line follows from Eq.~\eqref{trace_one_A}. On the RHS of the equation above, the integral over $K$ vanishes in accordance with the definition of $\Pi^2_K$, the fact that
\begin{align*}
\mathrm{Curl}(q) \in \VtoM{\begin{bmatrix}
			Q^{k-1,k-1,k-2,k-2}\\[1.0ex]
			Q^{k-1,k-2,k-1,k-2}\\[1.0ex]
			Q^{k-1,k-2,k-2,k-1}\\[1.0ex]
			Q^{k-2,k-1,k-1,k-2}\\[1.0ex]
			Q^{k-2,k-1,k-2,k-1}\\[1.0ex]
			Q^{k-2,k-2,k-1,k-1}
			 \end{bmatrix}}, 
\end{align*}
and Eq.~\eqref{eq:tesseract2a}. Next, it can be easily checked that for any of the normals on $K$, $(n \times q)$ yields a skew-symmetric matrix with three unique entries. In turn, there is a map between this matrix and a 3-vector $r \in Q^{k-2,k-1,k-1}\times Q^{k-1,k-2,k-1}\times Q^{k-1,k-1,k-2}(\partial K)$, and therefore the integral over $\partial K$ will vanish in accordance with the definition of the interpolant $\Pi_K^2$ and Eq.~\eqref{eq:hex2}. From this, the volume degrees of  freedom for $d^{(2)} \Pi^2_K p$ and $\Pi_K^3d^{(2)}p$ agree.

In a similar fashion, let us consider the \emph{facet} degrees of freedom. If $\mathcal{F}$ is a facet, and $q\in Q^{k-1,k-1,k-1}(\mathcal{F})$ we get
\begin{align*}
    \int_{\mathcal{F}} \mathrm{Tr} \left(d^{(2)} \Pi^2_K p -  \Pi^3_Kd^{(2)}p\right) q \, dx &= \int_{\mathcal{F}} d^{(2)} \mathrm{Tr}\left(\Pi^2_K p - p\right) q \, dx = \int_{\mathcal{F}} \nabla \cdot \mathrm{Tr}\left(\Pi^2_K p - p\right) q \, dx\\[1.0ex]
    &=-\int_{\mathcal{F}} \nabla q\cdot \mathrm{Tr}\left(\Pi^2_K p - p\right) \, dx + \int_{\partial \mathcal{F}} q \nu \cdot \mathrm{Tr}\left(\Pi^2_K p - p\right) \, ds.  
\end{align*}
In the first term on the RHS above, $\nabla q \in Q^{k-2,k-1,k-1}\times Q^{k-1,k-2,k-1} \times Q^{k-1,k-1,k-2}(\mathcal{F}).$ As a result, the integral over $\mathcal{F}$ vanishes in accordance with Eq.~\eqref{eq:hex2} and the definition of the interpolant. The integral over $\partial \mathcal{F}$ vanishes as well, in accordance with Eq.~\eqref{eq:square-2} and the definition of the interpolant. This calculation shows that the facet degrees of freedom for $d^{(2)} \Pi^2_K p$ and $\Pi_K^3d^{(2)}p$ agree. With some additional straightforward calculations, the face degrees of freedom are readily seen to match as well.

The case of $s =3$ can be treated using similar calculations. In addition, the proof for the situation where $\widehat{K}$ is a reference pentatope or tetrahedral prism follows in an analogous fashion, and is omitted here.
\end{proof}

\section{Conclusion}

During the course of this paper, we have successfully identified a natural sequence of Sobolev spaces in four dimensions: H(grad), H(skwGrad), H(curl), H(div), and L$^2$. We have described this sequence using both the language of linear algebra (scalars, vectors, and matrices), as well as the language of differential forms (0-forms, 1-forms, 2-forms, 3-forms, and 4-forms). A complete set of proxies have been developed in order to conveniently switch between the differential forms and the linear algebra entities. We believe that this explicit mathematical infrastructure will help facilitate the construction of new finite element spaces, as well as auxiliary preconditioners. In regards to finite element spaces, we have used this infrastructure to develop a preliminary set of high-order, conforming, finite element spaces on tesseract elements. The spaces, along with their associated degrees of freedom have been explicitly stated in straightforward language that is targeted towards scientists and engineers. In addition, we have rigorously proven the theoretical properties of the degrees of freedom for our spaces, including their unisolvency. Furthermore, we have provided a complete description of mapping operators between the reference space and the physical space for our chosen element. 

It is our hope that the work in this article will assist practitioners as they implement four-dimensional finite elements on tesseracts, and will encourage further developments in the analysis and implementation of conforming four-dimensional finite elements.

\section*{Declaration of Competing Interests}

The authors declare that they have no known competing financial interests or personal relationships that could have appeared to influence the work reported in this paper.

% \section*{Acknowledgements}

% The authors would like to thank 

\section*{Funding}

This research did not receive any specific grant from funding agencies in the public, commercial, or not-for-profit sectors.

{\footnotesize\bibliography{technical-refs}}

\appendix

\section{Derivative Identities} \label{derivative_appendix}

In this section, we prove the derivative identities in Eqs.~\eqref{derivative_id_one} and \eqref{derivative_id_two}. Consider a generic 1-form
\begin{align*}
    \omega \in \Lambda^{1}(\Omega), \qquad \omega = \omega_{1} dx^{1} + \omega_{2} dx^{2} + \omega_{3} dx^{3} + \omega_{4} dx^{4}.
\end{align*}
We begin by computing $d^{(1)} \omega$
\begin{align}
    d^{(1)} \omega = d^{(1)} \omega_{1} \wedge dx^{1} + d^{(1)} \omega_{2} \wedge dx^{2} + d^{(1)} \omega_{3} \wedge dx^{3} + d^{(1)} \omega_{4} \wedge dx^{4}. \label{derivative_commute_one}
\end{align}
Next, from the definition of the exterior derivative, one obtains
\begin{align}
    d^{(1)} \omega_{i} = \frac{\partial \omega_{i}}{\partial x_1} dx^{1} + \frac{\partial \omega_{i}}{\partial x_2} dx^{2} + \frac{\partial \omega_{i}}{\partial x_3} dx^{3} + \frac{\partial \omega_{i}}{\partial x_4} dx^{4},  \label{derivative_commute_two}
\end{align}
for $i = 1, 2, 3, 4$. Upon substituting Eq.~\eqref{derivative_commute_two} into Eq.~\eqref{derivative_commute_one}, one obtains
\begin{align*}
    d^{(1)} \omega &= \left( \frac{\partial \omega_{2}}{\partial x_1} - \frac{\partial \omega_{1}}{\partial x_2} \right) dx^{1} \wedge dx^{2} + \left( \frac{\partial \omega_{3}}{\partial x_1} - \frac{\partial \omega_{1}}{\partial x_3} \right) dx^{1} \wedge dx^{3} \\[1.0ex]
    &+\left( \frac{\partial \omega_{4}}{\partial x_1} - \frac{\partial \omega_{1}}{\partial x_4} \right) dx^{1} \wedge dx^{4} + \left( \frac{\partial \omega_{3}}{\partial x_2} - \frac{\partial \omega_{2}}{\partial x_3} \right) dx^{2} \wedge dx^{3} \\[1.0ex]
    &+\left( \frac{\partial \omega_{4}}{\partial x_2} - \frac{\partial \omega_{2}}{\partial x_4} \right) dx^{2} \wedge dx^{4} + \left( \frac{\partial \omega_{4}}{\partial x_3} - \frac{\partial \omega_{3}}{\partial x_4} \right) dx^{3} \wedge dx^{4}. 
\end{align*}
It then follows that
\begin{align*}
    \Upsilon_{2} \left( d^{(1)} \omega \right) = \frac{1}{2}
   \begin{bmatrix}
        0 & \partial_{1} \omega_2 - \partial_{2} \omega_1 & \partial_{1} \omega_3 - \partial_{3} \omega_1 & \partial_{1} \omega_4 - \partial_{4} \omega_1 \\[1.0ex]
        \partial_{2} \omega_1 - \partial_{1} \omega_2 & 0 & \partial_{2} \omega_3 - \partial_{3} \omega_2 & \partial_{2} \omega_4 - \partial_{4} \omega_2 \\[1.0ex]
        \partial_{3} \omega_1 - \partial_{1} \omega_3 & \partial_{3} \omega_2 - \partial_{2} \omega_3 & 0 & \partial_{3} \omega_4 - \partial_{4} \omega_3 \\[1.0ex]
        \partial_{4} \omega_1 - \partial_{1} \omega_{4} & \partial_{4} \omega_2 - \partial_{2} \omega_4 & \partial_{4} \omega_3 - \partial_{3} \omega_4 & 0
   \end{bmatrix} = \text{skwGrad} \left( \Upsilon_{1} \omega \right).
\end{align*}
This completes the proof of Eq.~\eqref{derivative_id_one}.

Next, consider a generic 2-form
\begin{align*}
     \omega \in \Lambda^{2}(\Omega), \qquad \omega &= \omega_{12} dx^{1} \wedge dx^{2} +\omega_{13} dx^{1} \wedge dx^{3} + \omega_{14} dx^{1} \wedge dx^{4} \\
    & + \omega_{23} dx^{2} \wedge dx^{3} + \omega_{24} dx^{2} \wedge dx^{4} + \omega_{34} dx^{3} \wedge dx^{4}.
\end{align*}
We can compute $d^{(2)} \omega$ as follows
\begin{align}
    \nonumber d^{(2)} \omega &= d^{(2)}\omega_{12} \wedge dx^{1} \wedge dx^{2} +d^{(2)} \omega_{13} \wedge dx^{1} \wedge dx^{3} + d^{(2)} \omega_{14} \wedge dx^{1} \wedge dx^{4} \\
    & + d^{(2)} \omega_{23} \wedge dx^{2} \wedge dx^{3} + d^{(2)} \omega_{24} \wedge dx^{2} \wedge dx^{4} + d^{(2)} \omega_{34} \wedge dx^{3} \wedge dx^{4}. \label{derivative_commute_three}
\end{align}
Now, from the definition of the exterior derivative, one obtains
\begin{align}
    d^{(2)} \omega_{ij} = \frac{\partial \omega_{ij}}{\partial x_1} dx^{1} + \frac{\partial \omega_{ij}}{\partial x_2} dx^{2} + \frac{\partial \omega_{ij}}{\partial x_3} dx^{3} + \frac{\partial \omega_{ij}}{\partial x_4} dx^{4},  \label{derivative_commute_four}
\end{align}
where $1\leq i < j \leq 4$. By substituting Eq.~\eqref{derivative_commute_four} into Eq.~\eqref{derivative_commute_three}, we obtain
\begin{align*}
    d^{(2)} \omega &= \left( \frac{\partial \omega_{12}}{\partial x_3} - \frac{\partial \omega_{13}}{\partial x_2} + \frac{\partial \omega_{23}}{\partial x_1} \right) dx^{1} \wedge dx^{2} \wedge dx^{3} + \left( \frac{\partial \omega_{12}}{\partial x_4}  - \frac{\partial \omega_{14}}{\partial x_2} + \frac{\partial \omega_{24}}{\partial x_1} \right) dx^{1} \wedge dx^{2} \wedge dx^{4} \\[1.0ex]
    &+ \left( \frac{\partial \omega_{13}}{\partial x_4} - \frac{\partial \omega_{14}}{\partial x_3} + \frac{\partial \omega_{34}}{\partial x_1} \right) dx^{1} \wedge dx^{3} \wedge dx^{4} + \left( \frac{\partial \omega_{23}}{\partial x_4} - \frac{\partial \omega_{24}}{\partial x_3} + \frac{\partial \omega_{34}}{\partial x_2} \right) dx^{2} \wedge dx^{3} \wedge dx^{4}.
\end{align*}
It then follows that
\begin{align*}
    \Upsilon_{3} \left( d^{(2)} \omega \right) = \begin{bmatrix}
    \partial_{4} \omega_{23} - \partial_{3} \omega_{24} + \partial_{2} \omega_{34} \\[1.0ex]
    -\partial_{4} \omega_{13} + \partial_{3} \omega_{14} - \partial_{1} \omega_{34} \\[1.0ex]
    \partial_{4} \omega_{12} - \partial_{2} \omega_{14} + \partial_{1} \omega_{24} \\[1.0ex]
    -\partial_{3} \omega_{12} + \partial_{2} \omega_{13} - \partial_{1} \omega_{23}
    \end{bmatrix} = \text{curl} \left( \Upsilon_{2} \omega \right).
\end{align*}
This completes the proof of Eq.~\eqref{derivative_id_two}.

\section{Details of the Pullback Construction} \label{pull_back_appendix}

In this section, we construct the pullback operator for $H \left(\text{curl}, \Omega, \mathbb{K} \right)$, the curl space in four dimensions, (see Eq.~\eqref{new_map_orig}). All of the other pullback operators are fairly standard, and do not require a detailed derivation. With this in mind, consider a generic 2-form
\begin{align*}
     \omega \in \Lambda^{2}(\Omega), \qquad \omega &= \omega_{12} dx^{1} \wedge dx^{2} +\omega_{13} dx^{1} \wedge dx^{3} + \omega_{14} dx^{1} \wedge dx^{4} \\
    & + \omega_{23} dx^{2} \wedge dx^{3} + \omega_{24} dx^{2} \wedge dx^{4} + \omega_{34} dx^{3} \wedge dx^{4}.
\end{align*}
We are interested in computing
\begin{align*}
     \phi^{\ast} \omega &= \phi^{\ast} \left( \omega_{12} dx^{1} \wedge dx^{2} \right) + \phi^{\ast} \left(\omega_{13} dx^{1} \wedge dx^{3} \right) + \phi^{\ast} \left(\omega_{14} dx^{1} \wedge dx^{4} \right) \\
    & + \phi^{\ast} \left( \omega_{23} dx^{2} \wedge dx^{3} \right) + \phi^{\ast} \left(\omega_{24} dx^{2} \wedge dx^{4} \right) + \phi^{\ast} \left(\omega_{34} dx^{3} \wedge dx^{4} \right).
\end{align*}
We note that, for example
\begin{align*}
    \phi^{\ast} \left( \omega_{12} dx^{1} \wedge dx^{2} \right) = \left(\omega_{12} \circ \phi \right) d\left(x_{1} \circ \phi \right) \wedge d\left(x_{2} \circ \phi\right).
\end{align*}
Based on this example, it is convenient to compute the following quantities
\begin{align*}
    d\left(x_{i}\circ \phi \right) &= \frac{\partial \phi_i}{\partial x_1} dx^1 + \frac{\partial \phi_i}{\partial x_2} dx^2 + \frac{\partial \phi_i}{\partial x_3} dx^3 + \frac{\partial \phi_i}{\partial x_4} dx^4,  
\end{align*}
where $i = 1,2,3,4$.
Using these identities, we obtain the following expressions
\begin{align}
    \nonumber & \phi^{\ast} \left( \omega_{12} dx^{1} \wedge dx^{2} \right) = \left(\omega_{12} \circ \phi \right) d\left(x_{1} \circ \phi \right) \wedge d\left(x_{2} \circ \phi\right) \\[1.0ex]
    \nonumber &= \left(\omega_{12} \circ \phi \right)\Bigg( \frac{\partial \phi_1}{\partial x_1} \frac{\partial \phi_2}{\partial x_2} dx^1 \wedge dx^2 + \frac{\partial \phi_1}{\partial x_1} \frac{\partial \phi_2}{\partial x_3} dx^1 \wedge dx^3 + \frac{\partial \phi_1}{\partial x_1} \frac{\partial \phi_2}{\partial x_4} dx^1 \wedge dx^4 + \frac{\partial \phi_1}{\partial x_2} \frac{\partial \phi_2}{\partial x_1} dx^2 \wedge dx^1 \\[1.0ex]
    \nonumber &+\frac{\partial \phi_1}{\partial x_2} \frac{\partial \phi_2}{\partial x_3} dx^2 \wedge dx^3 + \frac{\partial \phi_1}{\partial x_2} \frac{\partial \phi_2}{\partial x_4} dx^2 \wedge dx^4 + \frac{\partial \phi_1}{\partial x_3} \frac{\partial \phi_2}{\partial x_1} dx^3 \wedge dx^1 + \frac{\partial \phi_1}{\partial x_3} \frac{\partial \phi_2}{\partial x_2} dx^3 \wedge dx^2 \\[1.0ex]
    &+ \frac{\partial \phi_1}{\partial x_3} \frac{\partial \phi_2}{\partial x_4} dx^3 \wedge dx^4 + \frac{\partial \phi_1}{\partial x_4} \frac{\partial \phi_2}{\partial x_1} dx^4 \wedge dx^1 + \frac{\partial \phi_1}{\partial x_4} \frac{\partial \phi_2}{\partial x_2} dx^4 \wedge dx^2 + \frac{\partial \phi_1}{\partial x_4} \frac{\partial \phi_2}{\partial x_3} dx^4 \wedge dx^3 \Bigg). 
\end{align}

\begin{align}
    \nonumber &\phi^{\ast} \left(\omega_{13} dx^{1} \wedge dx^{3} \right) = \left(\omega_{13} \circ \phi \right) d\left(x_{1} \circ \phi \right) \wedge d\left(x_{3} \circ \phi\right) \\[1.0ex]
    \nonumber &=\left(\omega_{13} \circ \phi \right) \Bigg(\frac{\partial \phi_1}{\partial x_1} \frac{\partial \phi_3}{\partial x_2} dx^1 \wedge dx^2 + \frac{\partial \phi_1}{\partial x_1} \frac{\partial \phi_3}{\partial x_3} dx^1 \wedge dx^3 + \frac{\partial \phi_1}{\partial x_1} \frac{\partial \phi_3}{\partial x_4} dx^1 \wedge dx^4 + \frac{\partial \phi_1}{\partial x_2} \frac{\partial \phi_3}{\partial x_1} dx^2 \wedge dx^1 \\[1.0ex]
    \nonumber &+\frac{\partial \phi_1}{\partial x_2} \frac{\partial \phi_3}{\partial x_3} dx^2 \wedge dx^3 + \frac{\partial \phi_1}{\partial x_2} \frac{\partial \phi_3}{\partial x_4} dx^2 \wedge dx^4 + \frac{\partial \phi_1}{\partial x_3} \frac{\partial \phi_3}{\partial x_1} dx^3 \wedge dx^1 + \frac{\partial \phi_1}{\partial x_3} \frac{\partial \phi_3}{\partial x_2} dx^3 \wedge dx^2 \\[1.0ex]
    &+\frac{\partial \phi_1}{\partial x_3} \frac{\partial \phi_3}{\partial x_4} dx^3 \wedge dx^4 + \frac{\partial \phi_1}{\partial x_4} \frac{\partial \phi_3}{\partial x_1} dx^4 \wedge dx^1 + \frac{\partial \phi_1}{\partial x_4} \frac{\partial \phi_3}{\partial x_2} dx^4 \wedge dx^2 + \frac{\partial \phi_1}{\partial x_4} \frac{\partial \phi_3}{\partial x_3} dx^4 \wedge dx^3 \Bigg).
\end{align}

\begin{align}
    \nonumber &\phi^{\ast} \left(\omega_{14} dx^{1} \wedge dx^{4} \right) = \left(\omega_{14} \circ \phi \right) d\left(x_{1} \circ \phi \right) \wedge d\left(x_{4} \circ \phi\right) \\[1.0ex]
    \nonumber &=\left(\omega_{14} \circ \phi \right) \Bigg(\frac{\partial \phi_1}{\partial x_1} \frac{\partial \phi_4}{\partial x_2} dx^1 \wedge dx^2 + \frac{\partial \phi_1}{\partial x_1} \frac{\partial \phi_4}{\partial x_3} dx^1 \wedge dx^3 + \frac{\partial \phi_1}{\partial x_1} \frac{\partial \phi_4}{\partial x_4} dx^1 \wedge dx^4 + \frac{\partial \phi_1}{\partial x_2} \frac{\partial \phi_4}{\partial x_1} dx^2 \wedge dx^1 \\[1.0ex]
    \nonumber &+ \frac{\partial \phi_1}{\partial x_2} \frac{\partial \phi_4}{\partial x_3} dx^2 \wedge dx^3 + \frac{\partial \phi_1}{\partial x_2} \frac{\partial \phi_4}{\partial x_4} dx^2 \wedge dx^4 + \frac{\partial \phi_1}{\partial x_3} \frac{\partial \phi_4}{\partial x_1} dx^3 \wedge dx^1 + \frac{\partial \phi_1}{\partial x_3} \frac{\partial \phi_4}{\partial x_2} dx^3 \wedge dx^2 \\[1.0ex]
    &+ \frac{\partial \phi_1}{\partial x_3} \frac{\partial \phi_4}{\partial x_4} dx^3 \wedge dx^4 + \frac{\partial \phi_1}{\partial x_4} \frac{\partial \phi_4}{\partial x_1} dx^4 \wedge dx^1 + \frac{\partial \phi_1}{\partial x_4} \frac{\partial \phi_4}{\partial x_2} dx^4 \wedge dx^2 + \frac{\partial \phi_1}{\partial x_4} \frac{\partial \phi_4}{\partial x_3} dx^4 \wedge dx^3 \Bigg).
\end{align}

\begin{align}
    \nonumber &\phi^{\ast} \left( \omega_{23} dx^{2} \wedge dx^{3} \right) = \left(\omega_{23} \circ \phi \right) d\left(x_{2} \circ \phi \right) \wedge d\left(x_{3} \circ \phi\right) \\[1.0ex]
    \nonumber &= \left(\omega_{23} \circ \phi \right) \Bigg(\frac{\partial \phi_2}{\partial x_1} \frac{\partial \phi_3}{\partial x_2} dx^1 \wedge dx^2 + \frac{\partial \phi_2}{\partial x_1} \frac{\partial \phi_3}{\partial x_3} dx^1 \wedge dx^3 + \frac{\partial \phi_2}{\partial x_1} \frac{\partial \phi_3}{\partial x_4} dx^1 \wedge dx^4 + \frac{\partial \phi_2}{\partial x_2} \frac{\partial \phi_3}{\partial x_1} dx^2 \wedge dx^1 \\[1.0ex]
    \nonumber & + \frac{\partial \phi_2}{\partial x_2} \frac{\partial \phi_3}{\partial x_3} dx^2 \wedge dx^3 + \frac{\partial \phi_2}{\partial x_2} \frac{\partial \phi_3}{\partial x_4} dx^2 \wedge dx^4 + \frac{\partial \phi_2}{\partial x_3} \frac{\partial \phi_3}{\partial x_1} dx^3 \wedge dx^1 + \frac{\partial \phi_2}{\partial x_3} \frac{\partial \phi_3}{\partial x_2} dx^3 \wedge dx^2 \\[1.0ex]
    \nonumber &+ \frac{\partial \phi_2}{\partial x_3} \frac{\partial \phi_3}{\partial x_4} dx^3 \wedge dx^4 + \frac{\partial \phi_2}{\partial x_4} \frac{\partial \phi_3}{\partial x_1} dx^4 \wedge dx^1 + \frac{\partial \phi_2}{\partial x_4} \frac{\partial \phi_3}{\partial x_2} dx^4 \wedge dx^2 + \frac{\partial \phi_2}{\partial x_4} \frac{\partial \phi_3}{\partial x_3} dx^4 \wedge dx^3 \Bigg).
\end{align}

\begin{align}
    \nonumber &\phi^{\ast} \left(\omega_{24} dx^{2} \wedge dx^{4} \right) = \left(\omega_{24} \circ \phi \right) d\left(x_{2} \circ \phi \right) \wedge d\left(x_{4} \circ \phi\right) \\[1.0ex]
    \nonumber &= \left(\omega_{24} \circ \phi \right) \Bigg(\frac{\partial \phi_2}{\partial x_1} \frac{\partial \phi_4}{\partial x_2} dx^1 \wedge dx^2 + \frac{\partial \phi_2}{\partial x_1} \frac{\partial \phi_4}{\partial x_3} dx^1 \wedge dx^3 + \frac{\partial \phi_2}{\partial x_1} \frac{\partial \phi_4}{\partial x_4} dx^1 \wedge dx^4 + \frac{\partial \phi_2}{\partial x_2} \frac{\partial \phi_4}{\partial x_1} dx^2 \wedge dx^1 \\[1.0ex]
    \nonumber &+ \frac{\partial \phi_2}{\partial x_2} \frac{\partial \phi_4}{\partial x_3} dx^2 \wedge dx^3 + \frac{\partial \phi_2}{\partial x_2} \frac{\partial \phi_4}{\partial x_4} dx^2 \wedge dx^4 + \frac{\partial \phi_2}{\partial x_3} \frac{\partial \phi_4}{\partial x_1} dx^3 \wedge dx^1 + \frac{\partial \phi_2}{\partial x_3} \frac{\partial \phi_4}{\partial x_2} dx^3 \wedge dx^2 \\[1.0ex]
    &+ \frac{\partial \phi_2}{\partial x_3} \frac{\partial \phi_4}{\partial x_4} dx^3 \wedge dx^4 + \frac{\partial \phi_2}{\partial x_4} \frac{\partial \phi_4}{\partial x_1} dx^4 \wedge dx^1 + \frac{\partial \phi_2}{\partial x_4} \frac{\partial \phi_4}{\partial x_2} dx^4 \wedge dx^2 + \frac{\partial \phi_2}{\partial x_4} \frac{\partial \phi_4}{\partial x_3} dx^4 \wedge dx^3 \Bigg).
\end{align}

\begin{align}
    \nonumber &\phi^{\ast} \left(\omega_{34} dx^{3} \wedge dx^{4} \right) = \left(\omega_{34} \circ \phi \right) d\left(x_{3} \circ \phi \right) \wedge d\left(x_{4} \circ \phi\right) \\[1.0ex]
    \nonumber & = \left(\omega_{34} \circ \phi \right) \Bigg(\frac{\partial \phi_3}{\partial x_1} \frac{\partial \phi_4}{\partial x_2} dx^1 \wedge dx^2 + \frac{\partial \phi_3}{\partial x_1} \frac{\partial \phi_4}{\partial x_3} dx^1 \wedge dx^3 + \frac{\partial \phi_3}{\partial x_1} \frac{\partial \phi_4}{\partial x_4} dx^1 \wedge dx^4 + \frac{\partial \phi_3}{\partial x_2} \frac{\partial \phi_4}{\partial x_1} dx^2 \wedge dx^1 \\[1.0ex]
    \nonumber &+ \frac{\partial \phi_3}{\partial x_2} \frac{\partial \phi_4}{\partial x_3} dx^2 \wedge dx^3 + \frac{\partial \phi_3}{\partial x_2} \frac{\partial \phi_4}{\partial x_4} dx^2 \wedge dx^4 + \frac{\partial \phi_3}{\partial x_3} \frac{\partial \phi_4}{\partial x_1} dx^3 \wedge dx^1 + \frac{\partial \phi_3}{\partial x_3} \frac{\partial \phi_4}{\partial x_2} dx^3 \wedge dx^2 \\[1.0ex]
    &+ \frac{\partial \phi_3}{\partial x_3} \frac{\partial \phi_4}{\partial x_4} dx^3 \wedge dx^4 + \frac{\partial \phi_3}{\partial x_4} \frac{\partial \phi_4}{\partial x_1} dx^4 \wedge dx^1 + \frac{\partial \phi_3}{\partial x_4} \frac{\partial \phi_4}{\partial x_2} dx^4 \wedge dx^2 + \frac{\partial \phi_3}{\partial x_4} \frac{\partial \phi_4}{\partial x_3} dx^4 \wedge dx^3 \Bigg).
\end{align}

Thereafter, we can show that
\begin{align}
    \nonumber \phi^{\ast} \omega &= \Bigg[\left(\omega_{12} \circ \phi \right) \left( \frac{\partial \phi_1}{\partial x_1} \frac{\partial \phi_2}{\partial x_2}  - \frac{\partial \phi_2}{\partial x_1} \frac{\partial \phi_1}{\partial x_2} \right) + \left(\omega_{13} \circ \phi \right) \left(\frac{\partial \phi_1}{\partial x_1} \frac{\partial \phi_3}{\partial x_2}  - \frac{\partial \phi_3}{\partial x_1} \frac{\partial \phi_1}{\partial x_2} \right) \\[1.0ex] 
    \nonumber &+\left(\omega_{14} \circ \phi \right) \left( \frac{\partial \phi_1}{\partial x_1} \frac{\partial \phi_4}{\partial x_2}  - \frac{\partial \phi_4}{\partial x_1} \frac{\partial \phi_1}{\partial x_2} \right) + \left(\omega_{23} \circ \phi \right) \left(\frac{\partial \phi_2}{\partial x_1} \frac{\partial \phi_3}{\partial x_2}  - \frac{\partial \phi_3}{\partial x_1} \frac{\partial \phi_2}{\partial x_2} \right) \\[1.0ex]
    \nonumber &+\left(\omega_{24} \circ \phi \right) \left( \frac{\partial \phi_2}{\partial x_1} \frac{\partial \phi_4}{\partial x_2}  - \frac{\partial \phi_4}{\partial x_1} \frac{\partial \phi_2}{\partial x_2} \right) + \left(\omega_{34} \circ \phi \right) \left(\frac{\partial \phi_3}{\partial x_1} \frac{\partial \phi_4}{\partial x_2}  - \frac{\partial \phi_4}{\partial x_1} \frac{\partial \phi_3}{\partial x_2} \right)\Bigg] dx^1 \wedge dx^2 \\[1.0ex]
    \nonumber &+\Bigg[\left(\omega_{12} \circ \phi \right) \left( \frac{\partial \phi_1}{\partial x_1} \frac{\partial \phi_2}{\partial x_3}  - \frac{\partial \phi_2}{\partial x_1} \frac{\partial \phi_1}{\partial x_3} \right) + \left(\omega_{13} \circ \phi \right) \left(\frac{\partial \phi_1}{\partial x_1} \frac{\partial \phi_3}{\partial x_3}  - \frac{\partial \phi_3}{\partial x_1} \frac{\partial \phi_1}{\partial x_3} \right) \\[1.0ex] 
    \nonumber &+\left(\omega_{14} \circ \phi \right) \left( \frac{\partial \phi_1}{\partial x_1} \frac{\partial \phi_4}{\partial x_3}  - \frac{\partial \phi_4}{\partial x_1} \frac{\partial \phi_1}{\partial x_3} \right) + \left(\omega_{23} \circ \phi \right) \left(\frac{\partial \phi_2}{\partial x_1} \frac{\partial \phi_3}{\partial x_3}  - \frac{\partial \phi_3}{\partial x_1} \frac{\partial \phi_2}{\partial x_3} \right) \\[1.0ex]
    \nonumber &+\left(\omega_{24} \circ \phi \right) \left( \frac{\partial \phi_2}{\partial x_1} \frac{\partial \phi_4}{\partial x_3}  - \frac{\partial \phi_4}{\partial x_1} \frac{\partial \phi_2}{\partial x_3} \right) + \left(\omega_{34} \circ \phi \right) \left(\frac{\partial \phi_3}{\partial x_1} \frac{\partial \phi_4}{\partial x_3}  - \frac{\partial \phi_4}{\partial x_1} \frac{\partial \phi_3}{\partial x_3} \right)\Bigg] dx^1 \wedge dx^3  \\[1.0ex]
    \nonumber &+\Bigg[\left(\omega_{12} \circ \phi \right) \left( \frac{\partial \phi_1}{\partial x_1} \frac{\partial \phi_2}{\partial x_4}  - \frac{\partial \phi_2}{\partial x_1} \frac{\partial \phi_1}{\partial x_4} \right) + \left(\omega_{13} \circ \phi \right) \left(\frac{\partial \phi_1}{\partial x_1} \frac{\partial \phi_3}{\partial x_4}  - \frac{\partial \phi_3}{\partial x_1} \frac{\partial \phi_1}{\partial x_4} \right) \\[1.0ex] 
    \nonumber &+\left(\omega_{14} \circ \phi \right) \left( \frac{\partial \phi_1}{\partial x_1} \frac{\partial \phi_4}{\partial x_4}  - \frac{\partial \phi_4}{\partial x_1} \frac{\partial \phi_1}{\partial x_4} \right) + \left(\omega_{23} \circ \phi \right) \left(\frac{\partial \phi_2}{\partial x_1} \frac{\partial \phi_3}{\partial x_4}  - \frac{\partial \phi_3}{\partial x_1} \frac{\partial \phi_2}{\partial x_4} \right) \\[1.0ex]
    \nonumber &+\left(\omega_{24} \circ \phi \right) \left( \frac{\partial \phi_2}{\partial x_1} \frac{\partial \phi_4}{\partial x_4}  - \frac{\partial \phi_4}{\partial x_1} \frac{\partial \phi_2}{\partial x_4} \right) + \left(\omega_{34} \circ \phi \right) \left(\frac{\partial \phi_3}{\partial x_1} \frac{\partial \phi_4}{\partial x_4}  - \frac{\partial \phi_4}{\partial x_1} \frac{\partial \phi_3}{\partial x_4} \right)\Bigg] dx^1 \wedge dx^4  \\[1.0ex]
    \nonumber &+\Bigg[\left(\omega_{12} \circ \phi \right) \left( \frac{\partial \phi_1}{\partial x_2} \frac{\partial \phi_2}{\partial x_3}  - \frac{\partial \phi_2}{\partial x_2} \frac{\partial \phi_1}{\partial x_3} \right) + \left(\omega_{13} \circ \phi \right) \left(\frac{\partial \phi_1}{\partial x_2} \frac{\partial \phi_3}{\partial x_3}  - \frac{\partial \phi_3}{\partial x_2} \frac{\partial \phi_1}{\partial x_3} \right) \\[1.0ex] 
    \nonumber &+\left(\omega_{14} \circ \phi \right) \left( \frac{\partial \phi_1}{\partial x_2} \frac{\partial \phi_4}{\partial x_3}  - \frac{\partial \phi_4}{\partial x_2} \frac{\partial \phi_1}{\partial x_3} \right) + \left(\omega_{23} \circ \phi \right) \left(\frac{\partial \phi_2}{\partial x_2} \frac{\partial \phi_3}{\partial x_3}  - \frac{\partial \phi_3}{\partial x_2} \frac{\partial \phi_2}{\partial x_3} \right) \\[1.0ex]
    \nonumber &+\left(\omega_{24} \circ \phi \right) \left( \frac{\partial \phi_2}{\partial x_2} \frac{\partial \phi_4}{\partial x_3}  - \frac{\partial \phi_4}{\partial x_2} \frac{\partial \phi_2}{\partial x_3} \right) + \left(\omega_{34} \circ \phi \right) \left(\frac{\partial \phi_3}{\partial x_2} \frac{\partial \phi_4}{\partial x_3}  - \frac{\partial \phi_4}{\partial x_2} \frac{\partial \phi_3}{\partial x_3} \right)\Bigg] dx^2 \wedge dx^3 
\end{align}

\begin{align}
    \nonumber &+\Bigg[\left(\omega_{12} \circ \phi \right) \left( \frac{\partial \phi_1}{\partial x_2} \frac{\partial \phi_2}{\partial x_4}  - \frac{\partial \phi_2}{\partial x_2} \frac{\partial \phi_1}{\partial x_4} \right) + \left(\omega_{13} \circ \phi \right) \left(\frac{\partial \phi_1}{\partial x_2} \frac{\partial \phi_3}{\partial x_4}  - \frac{\partial \phi_3}{\partial x_2} \frac{\partial \phi_1}{\partial x_4} \right) \\[1.0ex] 
    \nonumber &+\left(\omega_{14} \circ \phi \right) \left( \frac{\partial \phi_1}{\partial x_2} \frac{\partial \phi_4}{\partial x_4}  - \frac{\partial \phi_4}{\partial x_2} \frac{\partial \phi_1}{\partial x_4} \right) + \left(\omega_{23} \circ \phi \right) \left(\frac{\partial \phi_2}{\partial x_2} \frac{\partial \phi_3}{\partial x_4}  - \frac{\partial \phi_3}{\partial x_2} \frac{\partial \phi_2}{\partial x_4} \right) \\[1.0ex]
    \nonumber &+\left(\omega_{24} \circ \phi \right) \left( \frac{\partial \phi_2}{\partial x_2} \frac{\partial \phi_4}{\partial x_4}  - \frac{\partial \phi_4}{\partial x_2} \frac{\partial \phi_2}{\partial x_4} \right) + \left(\omega_{34} \circ \phi \right) \left(\frac{\partial \phi_3}{\partial x_2} \frac{\partial \phi_4}{\partial x_4}  - \frac{\partial \phi_4}{\partial x_2} \frac{\partial \phi_3}{\partial x_4} \right)\Bigg] dx^2 \wedge dx^4 \\[1.0ex]
    \nonumber &+\Bigg[\left(\omega_{12} \circ \phi \right) \left( \frac{\partial \phi_1}{\partial x_3} \frac{\partial \phi_2}{\partial x_4}  - \frac{\partial \phi_2}{\partial x_3} \frac{\partial \phi_1}{\partial x_4} \right) + \left(\omega_{13} \circ \phi \right) \left(\frac{\partial \phi_1}{\partial x_3} \frac{\partial \phi_3}{\partial x_4}  - \frac{\partial \phi_3}{\partial x_3} \frac{\partial \phi_1}{\partial x_4} \right) \\[1.0ex] 
    \nonumber &+\left(\omega_{14} \circ \phi \right) \left( \frac{\partial \phi_1}{\partial x_3} \frac{\partial \phi_4}{\partial x_4}  - \frac{\partial \phi_4}{\partial x_3} \frac{\partial \phi_1}{\partial x_4} \right) + \left(\omega_{23} \circ \phi \right) \left(\frac{\partial \phi_2}{\partial x_3} \frac{\partial \phi_3}{\partial x_4}  - \frac{\partial \phi_3}{\partial x_3} \frac{\partial \phi_2}{\partial x_4} \right) \\[1.0ex]
    \nonumber &+\left(\omega_{24} \circ \phi \right) \left( \frac{\partial \phi_2}{\partial x_3} \frac{\partial \phi_4}{\partial x_4}  - \frac{\partial \phi_4}{\partial x_3} \frac{\partial \phi_2}{\partial x_4} \right) + \left(\omega_{34} \circ \phi \right) \left(\frac{\partial \phi_3}{\partial x_3} \frac{\partial \phi_4}{\partial x_4}  - \frac{\partial \phi_4}{\partial x_3} \frac{\partial \phi_3}{\partial x_4} \right)\Bigg] dx^3 \wedge dx^4.
\end{align}
It immediately follows that
\begin{align}
    \Upsilon_{2} \phi^{\ast} \omega = D\phi^{T} \left[F \circ \phi \right] D\phi,
\end{align}
where $F = \Upsilon_{2} \omega$.

\end{document}

%% file: guiding.tex
\section{Notation, Preliminaries, and Guiding Principles}

We begin this section by introducing the reference tesseract which will be used extensively throughout the paper.  We next recall some well-known degrees of freedom on hexahedra, which will be used extensively in our constructions. We end by discussing some guiding principles which will be used in our development of explicit, high-order, conforming families of finite element functions, which form exact sequences.

\subsection{The Reference Element}
Consider the following definition of a reference tesseract
\begin{align*}
    \widehat{K} := \C{4} := \left\{ x = \left(x_1, x_2, x_3, x_4 \right) \in \mathbb{R}^4 \big| -1 \leq x_1, x_2, x_3, x_4 \leq 1 \right\},
\end{align*}
with vertices
\begin{align*}
   v_{1} &= [
         -1 ,
         -1 ,
         -1 ,
         -1]^{T}, \quad 
    v_{2} = [
         1 ,
         -1 ,
         -1 ,
         -1
    ]^{T}, \quad 
    v_{3} = [
         1 ,
         1 ,
         -1 ,
         -1
    ]^{T}, \quad 
    v_{4} = [
         -1 ,
         1 ,
         -1 ,
         -1
    ]^{T}, \\[1.0ex]
    v_{5} &= [
         -1 ,
         -1 ,
         1 ,
         -1
    ]^{T}, \quad 
    v_{6} = [
         1 ,
         -1 ,
         1 ,
         -1
    ]^{T}, \quad 
    v_{7} = [
         1 ,
         1 ,
         1 ,
         -1
    ]^{T}, \quad 
    v_{8} = [
         -1 ,
         1 ,
         1 ,
         -1
    ]^{T}, \\[1.0ex]
    v_{9} &= [
         -1 ,
         -1 ,
         -1 ,
         1
    ]^{T}, \quad 
    v_{10} = [
         1 ,
         -1 ,
         -1 ,
         1
    ]^{T}, \quad 
    v_{11} = [
         1 ,
         1 ,
         -1 ,
         1
    ]^{T}, \quad 
    v_{12} = [
         -1 ,
         1 ,
         -1 ,
         1
    ]^{T}, \\[1.0ex]
    v_{13} &= [
         -1 ,
         -1 ,
         1 ,
         1
    ]^{T}, \quad 
    v_{14} = [
         1 ,
         -1 ,
         1 ,
         1
    ]^{T}, \quad 
    v_{15} = [
         1 ,
         1 ,
         1 ,
         1
    ]^{T}, \quad 
    v_{16} = [
         -1 ,
         1 ,
         1 ,
         1
    ]^{T}.
\end{align*}
Next, we introduce the definition of an arbitrary tesseract $K$ with vertices $v_{1}', v_{2}', \ldots, v_{16}'$.
There exists a bijective mapping between the reference tesseract and the arbitrary tesseract $\phi: \widehat{K} \rightarrow K$, such that
\begin{align*}
    \phi \left(x_1, x_2, x_3, x_4 \right) = \sum_{i=1}^{16} v_{i}' N_{i}\left(x_1, x_2, x_3, x_4 \right),
\end{align*}
where 
\begin{align*}
    N_1 &= \frac{1}{16} \left(1-x_{1}\right)\left(1-x_{2}\right)\left(1-x_{3}\right)\left(1-x_{4}\right), \quad     N_2 = \frac{1}{16} \left(1+x_{1}\right)\left(1-x_{2}\right)\left(1-x_{3}\right)\left(1-x_{4}\right), \\
    N_3 &= \frac{1}{16} \left(1+x_{1}\right)\left(1+x_{2}\right)\left(1-x_{3}\right)\left(1-x_{4}\right), \quad     N_4 = \frac{1}{16} \left(1-x_{1}\right)\left(1+x_{2}\right)\left(1-x_{3}\right)\left(1-x_{4}\right), \\
    N_5 &= \frac{1}{16} \left(1-x_{1}\right)\left(1-x_{2}\right)\left(1+x_{3}\right)\left(1-x_{4}\right), \quad     N_6 = \frac{1}{16} \left(1+x_{1}\right)\left(1-x_{2}\right)\left(1+x_{3}\right)\left(1-x_{4}\right), \\
    N_7 &= \frac{1}{16} \left(1+x_{1}\right)\left(1+x_{2}\right)\left(1+x_{3}\right)\left(1-x_{4}\right), \quad     N_8 = \frac{1}{16} \left(1-x_{1}\right)\left(1+x_{2}\right)\left(1+x_{3}\right)\left(1-x_{4}\right), \\
    N_9 &= \frac{1}{16} \left(1-x_{1}\right)\left(1-x_{2}\right)\left(1-x_{3}\right)\left(1+x_{4}\right), \quad     N_{10} = \frac{1}{16} \left(1+x_{1}\right)\left(1-x_{2}\right)\left(1-x_{3}\right)\left(1+x_{4}\right), \\
    N_{11} &= \frac{1}{16} \left(1+x_{1}\right)\left(1+x_{2}\right)\left(1-x_{3}\right)\left(1+x_{4}\right), \quad     N_{12} = \frac{1}{16} \left(1-x_{1}\right)\left(1+x_{2}\right)\left(1-x_{3}\right)\left(1+x_{4}\right), \\
    N_{13} &= \frac{1}{16} \left(1-x_{1}\right)\left(1-x_{2}\right)\left(1+x_{3}\right)\left(1+x_{4}\right), \quad     N_{14} = \frac{1}{16} \left(1+x_{1}\right)\left(1-x_{2}\right)\left(1+x_{3}\right)\left(1+x_{4}\right), \\
    N_{15} &= \frac{1}{16} \left(1+x_{1}\right)\left(1+x_{2}\right)\left(1+x_{3}\right)\left(1+x_{4}\right), \quad     N_{16} = \frac{1}{16} \left(1-x_{1}\right)\left(1+x_{2}\right)\left(1+x_{3}\right)\left(1+x_{4}\right).
\end{align*}
Figure~\ref{fig:tesseract} illustrates a generic tesseract.
\begin{figure}[h!]
    \centering
    \includegraphics[width=10cm]{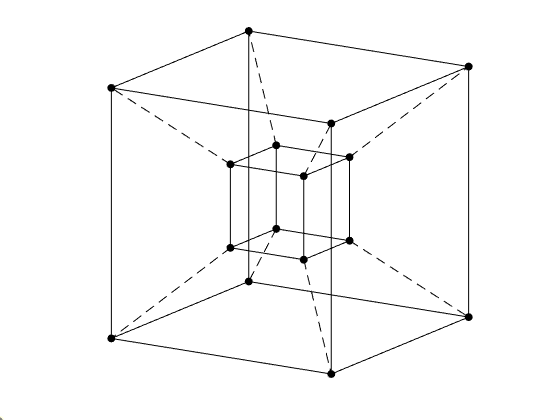}
    \caption{Illustration of a generic tesseract.}
    \label{fig:tesseract}
\end{figure}
We summarize important geometric information associated with the tesseract in \autoref{Table:4DelementsSubs}. Here, the $d$-dimensional reference cube is denoted by $\C{d}$.
\begin{table}[h!]
    \centering
    \begin{tabular}{c|c}\hline
   & Tesseract \\
    & $\C{4}$\\ \hline
    Vertices&16\\ \hline
    Edges&32\\ \hline
    Quadrilateral faces $\C{2}$&24\\ \hline
    Hexahedral facets $\C{3}$&8 
    \end{tabular}
    \caption{Geometric information for the reference tesseract.}
    \label{Table:4DelementsSubs}
\end{table}

\subsection{Notation}  We will denote by $\Vk{k}{s}(\Omega)$ the space of $k$-th order polynomial shape functions for the $s$-forms on~$\Omega$. $\Sk{k}{s}(\Omega)$ denotes the {\it degrees of freedom (dofs),} a collection of linear functionals on $\Vk{k}{s}(\Omega)$ which is dual to this polynomial space. In this paper, we will be presenting explicit descriptions of finite element triples
\begin{align*}
\left(\widehat{K},\Vk{k}{s}(\widehat{K}), \Sk{k}{s}(\widehat{K})\right),
\end{align*}
for reference element $\widehat{K}= \C{4}$. We suppress the dependence on $\widehat{K}$ when the context is clear.

We recall that $P^k(x_1,x_2,x_3,x_4)$ denotes the space of polynomials of maximal {\it total} degree $k$ in the variables $x_1,x_2,x_3,x_4$. We denote the space of homogeneous polynomials of total degree exactly $k$ in these variables by $\tilde{P}^k$. That is, if we set $x=(x_1,x_2,x_3,x_4)$ then
\begin{align*}
\sum_{|\alpha| \leq k} a_{\alpha} x^\alpha \in P^k(x_1,x_2,x_3,x_4), \qquad  \sum_{|\alpha| =  k} a_{\alpha} x^\alpha \in \tilde{P}^k(x_1,x_2,x_3,x_4),
\end{align*}
where $\alpha$ is the multi-index, and $a_{\alpha}$ are constants. We will suppress the arguments $(x_1,x_2,x_3,x_4)$ when the context is clear. The notation $Q^{l,m,n,q}(x_1,x_2,x_3,x_4)$ denotes standard tensorial polynomials of maximal degree $l,m,n,q$, that is,
\begin{align*}
 Q^{l,m,n,q}(x_1,x_2,x_3,x_4)=P^l(x_1)P^m(x_2)P^n(x_3)P^q(x_4).
\end{align*}

Next, consider a bijective map from 6-vectors to skew-symmetric matrices  in $\mathbb{K}$:
\begin{align}\label{eq:Ldefine}
\VtoM{\cdot}:\mathbb{R}^6 \rightarrow \mathbb{K} : \VtoM{\begin{bmatrix}
	w_{12}\\
	w_{13}\\
	w_{14}\\
	w_{23}\\
	w_{24}\\
	w_{34}
	\end{bmatrix}}:= \begin{bmatrix}
0&w_{12}&w_{13}&w_{14}\\
-w_{12}&0&w_{23}&w_{24}\\
-w_{13}&-w_{23}&0&w_{34}\\
-w_{14}&-w_{24}&-w_{34}&0
\end{bmatrix}.
\end{align}
Finally, we introduce the following pair of operators that denote the trace of a quantity $u$ on to a $n$-dimensional submanifold $f$:
\begin{align*}
    \tr{f}{u}, \qquad \Tr{f}{u}.
\end{align*}
The argument $[f]$ is omitted when the submanifold in question is obvious. The first trace operator $\tr{f}{u}$ denotes the well-defined restriction of $u$ on to $f$, where the restriction is a scalar, $4$-vector, or $4 \times 4$ matrix. In a similar fashion, the second trace operator $\Tr{f}{u}$ denotes the well-defined restriction of $u$ on to $f$, where the restriction is a scalar, $n$-vector, or $n \times n$ matrix. Evidently, these operators are identical when $n = 4$. In addition, there is (at least) a surjective map between elements in the ranges of the operators
\begin{align*}
    \Xi: \mathrm{tr}[f](u) \longrightarrow \mathrm{Tr}[f](u).
\end{align*}
In other words, information from $\mathrm{tr}[f](u)$, which is a scalar, 4-vector, or $4\times 4$ matrix can always be identified with $\mathrm{Tr}[f](u)$, which is a scalar, $n$-vector, or $n \times n$ matrix.

\subsection{Degrees of Freedom}
In later sections of this paper, our construction of finite element triples will rely heavily on the use of well-known dofs on hexahedral {\it facets}, and their associated faces, edges, and vertices. For the sake of completeness, we record these dofs below. 
 
\subsubsection{Vertex Degrees of Freedom}
Vertex degrees of freedom are only defined for 0-forms. For the tesseract, we specify the vertex degrees of freedom $\Sigma^{k,(0)}(v)$ as the vertex values of the polynomial 0-form. We note that there are 16 such degrees of freedom for the tesseract $\C{4}$.

\subsubsection{Edge Degrees of Freedom}
Edge degrees of freedom are only defined for 0- and 1-forms. Let $e$ be an edge of an element $\widehat{K}$, where $\widehat{K}$ = $\C{4}$, and let $u \in \Vk{k}{0}(\widehat{K})$ be a 0-form proxy. We define edge degrees of freedom for $u$ as follows
\begin{equation}\label{eq:edge-0}
M_{e}(u) :=\left\{ \int_{e} \text{Tr}(u) q, \, \qquad \forall q \in P^{k-2}(e), \qquad  \mbox{ for each edge } \, e \, \mbox{ of }\, \widehat{K} \right\}.
\end{equation}
Generally, for tesseracts $\C{4}$ there are $32(k-1)$ such degrees of freedom.

Now, let $U \in \Vk{k}{1}(\widehat{K})$ be a 1-form proxy. We define its edge degrees of freedom as
\begin{equation} \label{eq:edge-1}
M_{e}(U) :=\left\{ \int_{e} \mathrm{Tr} (U)\cdot \tau q, \, \qquad \forall q \in P^{k-1}(e), \qquad  \mbox{ for each edge } \, e \, \mbox{ of }\, \widehat{K} \right\},
\end{equation} 
where $\tau$ is a unit vector in the direction of $e$.
Generally, for $\C{4}$ there are $32k$ such degrees of freedom.

\subsubsection{Face Degrees of Freedom}
Face degrees of freedom are only defined for 0-, 1-, and 2-forms. Let $f$ denote a single face of an element $\widehat{K}$. This face will be a quadrilateral since $\widehat{K} = \C{4}$.

{\bf Quadrilateral faces:}
For polynomial 0-forms $u\in \Vk{k}{0}(\widehat{K})$, face degrees of freedom on a quadrilateral face $f = \C{2}$ are defined as
\begin{equation}\label{eq:square-0}
M_f(u):=\left\{ \int_f \text{Tr}(u)q, \quad \forall q \in Q^{k-2,k-2}(f) \right\}.
\end{equation}
Face degrees of freedom for polynomial 1-forms, $U \in \Vk{k}{1}(\widehat{K})$, are given by
\begin{equation}\label{eq:square-1}
M_f(U):=\left\{ \int_f (\mathrm{Tr}(U)\times \nu) \cdot q, \quad \forall q \in Q^{k-2,k-1}(f)\times Q^{k-1,k-2}(f) \right\}.
\end{equation}
Here, $\nu$ denotes a unit normal vector to the face $f$.
Lastly, face degrees of freedom for polynomial 2-forms, $U \in \Vk{k}{2}(\widehat{K})$, can be specified by
\begin{equation}\label{eq:square-2}
M_f(U):=\left\{ \int_f (\mathrm{Tr}(U)\cdot \nu) q, \quad \forall q \in Q^{k-1,k-1}(f) \right\}.
\end{equation}

\subsubsection{Facet Degrees of Freedom}
Facet degrees of freedom are defined for 0-, 1-, 2-, and 3-forms. The tesseract $\C{4}$ has hexahedral facets. 

{\bf Hexahedral facets:}
Let $\mathcal{F}=\C{3}$ denote a hexahedral facet.
For polynomial 0-forms $u \in \Vk{k}{0}(\widehat{K})$, we can specify facet degrees of freedom as
\begin{equation} M_{\mathcal{F}}(u) := \left\{ \int_{\mathcal{F}} \mathrm{Tr}(u) q, \qquad  q \in Q^{k-2,k-2,k-2}(\mathcal{F})\right\}.\label{eq:hex0} \end{equation}
For polynomial 1-forms $U \in \Vk{k}{1}(\widehat{K})$, we specify the facet degrees of freedom as
\begin{equation}\label{eq:hex1}
M_{\mathcal{F}}(U) := \left\{ \int_{\mathcal{F}} \mathrm{Tr}(U) \cdot  q, \qquad  q \in \left(Q^{k-1,k-2,k-2}\times Q^{k-2,k-1,k-2}\times Q^{k-2,k-2,k-1}\right)(\mathcal{F})\right\}.
\end{equation}
For polynomial 2-forms $U \in \Vk{k}{2}(\widehat{K})$, we specify the facet degrees of freedom as
\begin{equation}\label{eq:hex2}
M_{\mathcal{F}}(U) := \left\{ \int_{\mathcal{F}} \mathrm{Tr}(U) \cdot  q, \qquad  q \in \left(Q^{k-2,k-1,k-1}\times Q^{k-1,k-2,k-1}\times Q^{k-1,k-1,k-2}\right)(\mathcal{F})\right\}.
\end{equation}
Lastly, for polynomial 3-forms $U \in \Vk{k}{3}(\widehat{K})$, we specify the facet degrees of freedom as
\begin{equation}\label{eq:hex3}
M_{\mathcal{F}}(U) := \left\{ \int_{\mathcal{F}} \text{Tr}(U)  q, \qquad  q \in Q^{k-1,k-1,k-1}(\mathcal{F})\right\}.
\end{equation}

\subsection{Guiding Principles}

The design principles we shall follow in the process of constructing the polynomial spaces $\Vk{k}{s}(\K)$  are:
\begin{itemize}
	\item {\it Approximation.} Given a positive integer $p$, it is possible to choose $k\in \mathbb{N}$ so that all polynomial $s$-forms of degree $p$ are contained in $\Vk{k}{s}(\K).$
	\item {\it Compatibility.} The restriction of a finite element function in $\Vk{k}{s}(\Omega)$ to the facets $\mathcal{F}$ of $\Omega$ should match the traces of functions on neighbouring physical elements. 
	\item {\it Stability.} We require that the finite element spaces satisfy a suitable {\it commuting diagram} property.
\end{itemize} 
As we mentioned previously, $\Sk{k}{s}(\K)$ denotes the space of degrees of freedom. We insist that the dofs satisfy the follow properties: {\it unisolvence, invariance under `canonical' transformations, and locality}, (see~\cite{hiptmair1999canonical}). Unisolvency will follow by showing that if all the degrees of freedom  $\Sk{k}{s}(\K)$ vanish for some $u\in \Vk{k}{s}(\K)$, then $u\equiv 0$. Invariance means that, if $\K$ is mapped via an affine transformation to $\K'$, the degrees of freedom transform in the expected manner. Finally, locality simply means that if $\mathcal{F} \subset \K$ is a facet of $\K$, then the trace of $u\in \Vk{k}{s}(\K)$ on to $\mathcal{F}$ is specified by some degrees of freedom whose values only depend on this facet.  
Locality is a key component for {\it conformity:} degrees of freedom which possess trace values will be shared with neighbouring physical elements, and must be local to the shared submanifold.

With this in mind, suppose that a subdomain $f$  (vertex, edge, or face) is a subset of a facet $\mathcal{F}$, i.e.~$f \subset \mathcal{F}$ of $\K$. Let $u$ be a polynomial $s$-form.  Then computing the trace of $u$ on $f$, $\text{Tr}[f](u)$, is equivalent to computing the trace on to $f$, $\text{Tr}[f](\cdot)$ of the trace of $u$ on $\mathcal{F}$, $\text{Tr}[\mathcal{F}](u)$. Equivalently, 
\begin{align*}
\text{Tr}[f](u) = \text{Tr}[f]\left( \text{Tr}[\mathcal{F}](u)\right), 
\end{align*}
for $u\in \Vk{k}{s}(\K)$, provided such traces are defined. Therefore, given a description of degrees of freedom corresponding to $\Vk{k}{s}(\mathcal{F})$ on a facet, we do not need to provide any additional degrees of freedom for the associated vertices, edges, and faces -- these will be specified by well-known, lower-dimensional dofs. This allows us to decompose dofs on $\K$ in terms of {\it trace} dofs and {\it volumetric} dofs as follows
\begin{align*}
	\Sk{k}{s}(\K) = \Sk{k}{s}_{trace}(\K) \cup \Sk{k}{s}_{vol}(\K),
\end{align*}
where
\begin{align*}
	\Sk{k}{s}_{trace}(\K):= \bigcup_{r=s}^{3} \left(  \bigcup_{ \mathcal{F} \subset D_r} \Sk{k}{s}(\mathcal{F})  \right), 
\end{align*}
and where  $D_r$ denotes an $r$-dimensional submanifold of the boundary of $\K$.
Provided we pick facet degrees of freedom to ensure that they are unisolvent for facet polynomial $s$-forms in $\Vk{k}{s}(\mathcal{F})$, and provided
\begin{equation}
\text{Tr}[\mathcal{F}](\Vk{k}{s}(\K)) = \Vk{k}{s}(\mathcal{F}),
\end{equation}
we will guarantee conformity of our finite element triples. Note that we have unisolvence on a facet under the following circumstances: if all the facet degrees of freedom $\Sk{k}{s}(\mathcal{F}) =0$ for some $u \in \Vk{k}{s}(\K)$, then $\text{Tr}[\mathcal{F}](u)=0$. 

There are several different approaches for constructing volume degrees of freedom. Suppose $u \in \Vkcirc{k}{s}(\K)$ is a zero-trace polynomial $s$-form which belongs to the bubble space on an element $\K$. This form is uniquely characterized by specifying one of the following: 
\begin{itemize}
	\item Wedge products of the form
	\begin{equation*}
	    \left\{u \rightarrow \int_{\mathcal{F}} \mathrm{Tr}[\mathcal{F}](u)  \wedge q, \qquad \mbox{for polynomial} \,\, (4-s)\mbox{-forms}, q \right\}.
	\end{equation*}
	This approach is described in general for $d$-simplex elements in~\cite{arnold2006finite} and for $d$-cubes in~\cite{arnold2014finite}. 
	
	\item Projection-based degrees of freedom following the approach of \cite{monkbook, Demkowicz10}.
	
	\item Integration against certain polynomials. These polynomials are defined using explicit representations of functions in the bubble spaces. For example, if members $w\in \Vkcirc{k}{s}(\K)$ can be expressed in terms of products of the form $\psi_i q(x_1,x_2,x_3,x_4)$ where $\psi_i$ vanish on (parts of) the boundary, then degrees of freedom can be specified in terms of the polynomial objects~$q$.
	\item A discrete Helmholtz-like decomposition of $s$-form bubbles.
\end{itemize} 
Other approaches are also possible, although the authors are unaware of any which have the same popularity as those listed above. In the subsequent sections, we will make primary use of the third approach listed above due to its simplicity.

%% file: tesseract.tex
\section{Finite Elements on a Reference Tesseract} \label{tesseract_space}

We begin with an explicit construction of tensorial finite element approximation spaces for $s$-forms on the tesseract $\C{4}$, since the structure of these spaces is easy to see. These spaces have been discussed elsewhere \cite{arnold2011serendipity} for cubical meshes in $\mathbb{R}^n$. Our focus here is on explicit descriptions of these spaces and associated degrees of freedom. Another motivation is to highlight the description and role of the bubble spaces.

 In analogy with~\cite{monkbook} and noting that our spaces $\Vk{k}{s}(\C{4}) $ must satisfy the relation
\begin{align*}
	\begin{matrix}
		\Vk{k}{0}(\C{4}) & \arrow{r}{d^{\left(0\right)}} & \Vk{k}{1}(\C{4}) & \arrow{r}{d^{\left(1\right)}} & \Vk{k}{2}(\C{4}) & \arrow{r}{d^{\left(2\right)}} & \Vk{k}{3}(\C{4})  & \arrow{r}{d^{\left(3\right)}} & \Vk{k}{4}(\C{4}),
	\end{matrix}
\end{align*}
we readily see that 
\begin{subequations}
\begin{align}
		\Vk{k}{0}(\C{4}) &:= Q^{k,k,k,k} \label{eq:0formC}\\[1.0ex] 
		\Vk{k}{1}(\C{4}) &:= Q^{k-1,k,k,k}\times Q^{k,k-1,k,k}\times Q^{k,k,k-1,k}\times Q^{k,k,k,k-1}\label{eq:1formC}\\[1.0ex]
		\Vk{k}{2}(\C{4})&:= \VtoM{\begin{bmatrix}
				Q^{k-1,k-1,k,k}\\[1.0ex]
				Q^{k-1,k,k-1,k}\\[1.0ex]
				Q^{k-1,k,k,k-1}\\[1.0ex]
				Q^{k,k-1,k-1,k}\\[1.0ex]
				Q^{k,k-1,k,k-1}\\[1.0ex]
				Q^{k,k,k-1,k-1}
		\end{bmatrix}  } \label{eq:2formC}  \\[1.0ex]
		\Vk{k}{3}(\C{4}) &:= Q^{k,k-1,k-1,k-1}\times Q^{k-1,k,k-1,k-1}\times Q^{k-1,k-1,k,k-1} \times Q^{k-1,k-1,k-1,k}\label{eq:3formC}\\[1.0ex]
		\Vk{k}{4}(\C{4})&:= Q^{k-1,k-1,k-1,k-1},\label{eq:4formC}
\end{align} where we have used the operator $\VtoM{\cdot}$ defined in Eq.~\eqref{eq:Ldefine}.
\end{subequations}
It remains to identify the bubble spaces $\Vkcirc{k}{s}(\C{4})$. The tensorial nature of $\C{4}$ makes this an easy task:
%, using the trace definitions in Section 3.4.
\begin{subequations}
	\begin{align}
		\Vkcirc{k}{0}(\C{4}) &:= \text{span} \left\{\left(\Pi_{n=1}^4 (1-x_n^2)\right)p_i(x_1)p_j(x_2)p_{\ell}(x_3)p_{m}(x_4) \right\}, \quad \forall \; p_i, p_j, p_{\ell}, p_{m} \in P^{k-2}(\C{1}), \label{eq:0formCbubble}
		\end{align}
\begin{align}
	     \Vkcirc{k}{1}(\C{4}) &:= \text{span} \left\{\begin{bmatrix}
			\left(\Pi_{n=1, n\neq 1}^{4}(1- x_{n}^2)\right) p_i(x_1)p_j(x_2)p_{\ell}(x_3)p_{m}(x_4)\\[1.0ex]
                \left(\Pi_{n=1, n\neq 2}^{4}(1- x_{n}^2)\right) p_i(x_2)p_j(x_1)p_{\ell}(x_3)p_{m}(x_4)\\[1.0ex]
                \left(\Pi_{n=1, n\neq 3}^{4}(1- x_{n}^2)\right) p_i(x_3)p_j(x_1)p_{\ell}(x_2)p_{m}(x_4)\\[1.0ex]
			\left(\Pi_{n=1, n\neq 4}^{4}(1- x_{n}^2)\right) p_i(x_4)p_j(x_1)p_{\ell}(x_2)p_{m}(x_3) 
		\end{bmatrix} \right\}, \quad \forall p_i\in P^{k-1}(\C{1}), \quad p_{j},p_{\ell},p_{m} \in P^{k-2}(\C{1}),\label{eq:1formCbubble}  
\end{align}

\begin{align}
	     \Vkcirc{k}{2}(\C{4})&:= \text{span} \left\{ \VtoM{\begin{bmatrix}
				\left(\Pi_{n=1, n\neq1,2}^{4} (1-x_{n}^{2}) \right)                     p_i(x_1)p_j(x_2)p_{\ell}(x_3)p_{m}(x_4) \\[1.0ex]
                    \left(\Pi_{n=1, n\neq1,3}^{4} (1-x_{n}^{2}) \right) p_i(x_1)p_j(x_3)p_{\ell}(x_2)p_{m}(x_4) \\[1.0ex]
                    \left(\Pi_{n=1, n\neq1,4}^{4} (1-x_{n}^{2}) \right) p_i(x_1)p_j(x_4)p_{\ell}(x_2)p_{m}(x_3) \\[1.0ex]
                    \left(\Pi_{n=1, n\neq2,3}^{4} (1-x_{n}^{2}) \right) p_i(x_2)p_j(x_3)p_{\ell}(x_1)p_{m}(x_4) \\[1.0ex]
                    \left(\Pi_{n=1, n\neq2,4}^{4} (1-x_{n}^{2}) \right) p_i(x_2)p_j(x_4)p_{\ell}(x_1)p_{m}(x_3) \\[1.0ex]
                    \left(\Pi_{n=1, n\neq3,4}^{4} (1-x_{n}^{2}) \right) p_i(x_3)p_j(x_4)p_{\ell}(x_1)p_{m}(x_2)
		\end{bmatrix}  } \right\}, \label{eq:2formCbubble} \\[1.0ex]
		\nonumber &\forall \; p_{i}, p_{j} \in P^{k-1}(\C{1}), \qquad p_{\ell}, p_{m} \in P^{k-2}(\C{1}), \\[1.0ex]
		\Vkcirc{k}{3}(\C{4}) &:= \text{span} \left\{ \begin{bmatrix}
			(1-x_{1}^{2}) p_{i}(x_1)p_{j}(x_2)p_{\ell}(x_3)p_{m}(x_4) \\[1.0ex]
                (1-x_{2}^{2}) p_{i}(x_2)p_{j}(x_1)p_{\ell}(x_3)p_{m}(x_4) \\[1.0ex]
                (1-x_{3}^{2}) p_{i}(x_3)p_{j}(x_1)p_{\ell}(x_2)p_{m}(x_4) \\[1.0ex]
                (1-x_{4}^{2}) p_{i}(x_4)p_{j}(x_1)p_{\ell}(x_2)p_{m}(x_3)
		\end{bmatrix} \right\}, \label{eq:3formCbubble} \forall \; p_{i}\in P^{k-2}(\C{1}), \qquad p_{j}, p_{\ell}, p_{m} \in P^{k-1}(\C{1}). 
	\end{align}
\end{subequations}

%%%%%%%%%%
%%%%%%%%%%
 
\subsection{Degrees of Freedom on the Reference Tesseract, $\C{4}$ }
Almost a decade ago, a construction of conforming high-order finite element spaces on tensorial elements was provided in~\cite{arnold2014finite}. The degrees of freedom defined as part of this construction were given as
\begin{align*}
\left\{ u \rightarrow \int_f \text{Tr}[f]\left( u \right) \wedge v , \qquad v\in \mathcal{Q}_{k-1}^{-} \Lambda^{d-s}(f), \qquad f \in \Delta_d(\C{n}), \, k\geq 1, \, 0 \leq s\leq n, \, d \geq s \right\},
 \end{align*}
 where $\Delta_d(\C{n})$ are the set of $d$-dimensional faces of the $n$-dimensional reference cube. These are analogs of the $\mathcal{P}_{k}^{-}\Lambda^s$ spaces from \cite{arnold2010finite}. %A different construction was presented in \cite{quadfeec}, based on a tensorial construction.

In the present work, we specialize the discussion to tensor-product  elements in $\mathbb{R}^4$. Our goal is to provide explicit families of approximation spaces for the spaces in the de Rham sequence (Eq.~\eqref{eq:ourdehram}), and where possible demonstrate these using proxies rather than forms.
Instead of wedge products, we use explicitly-constructed bubble polynomial spaces $\Vkcirc{k}{s}(\C{4}).$ 

%The other approach mimics a projection-based interpolation construction, see~\cite{Demkowicz12}. 

We begin by noting that we can decompose the degrees of freedom as those associated with the traces of $\Vk{k}{s}(\C{4})$ and the so-called {\it volumetric} degrees of freedom:
\begin{align*}
\Sk{k}{s}(\C{4}):=\Sigma_{trace}^{k,s}(\C{4}) \cup \Sigma_{vol}^{k,s}(\C{4}).
\end{align*}
Additionally, following the construction for tesseracts, we use well-known trace degrees of freedom for polynomial $s$-forms on hexahedra, quadrilaterals, and line segments to specify $\Sigma_{trace}^{k,s}(\C{4})$. If $u\in \Vk{k}{s}(\C{4})$ has all its trace degrees of freedom vanish then $u \in \Vkcirc{k}{s}(\C{4}).$ Our task, then, is to provide well-defined dual spaces for $\Vkcirc{k}{s}(\C{4}).$

\subsection{Dofs for 0-forms on $\C{4}$}
We recall that $\Vk{k}{0}(\C{4}):=Q^{k,k,k,k}$ and hence
\begin{align*}
\text{dim}\left(\Vk{k}{0}(\C{4})\right) = \text{dim}\left( \Sigma^{k,0}(\C{4}) \right) = (k+1)^4.
\end{align*}
There are 16 vertices, 32 edges $\C{1}$, 24 quadrilateral faces $\C{2}$, and 8 cubic facets $\C{3}$. Using the degrees of freedom associated with these objects (defined in Eqs.~\eqref{eq:edge-0}, \eqref{eq:square-0}, and \eqref{eq:hex0}) we see that
\begin{align*}	
	\text{dim} \left(\Sigma_{trace}^{k,0}(\C{4})\right)&:= 16 + 32 \, \text{dim} (P^{k-2}(\C{1})) \\
    &+ 24 \, \text{dim} \left(Q^{k-2,k-2}(\C{2})\right) + 8 \,\text{dim}\left( Q^{k-2,k-2,k-2}(\C{3}) \right) \\
	&= 16 + 32(k-1)+24(k-1)^2 + 8 (k-1)^3 = 8k(k^2+1).
\end{align*}
From this, it follows that for 0-forms 
\begin{align*}\text{dim}\left(\Sigma_{vol}^{k,0}(\C{4})\right) &= \text{dim}\left(\Sigma^{k,0}(\C{4})\right) - \text{dim}\left(\Sigma_{trace}^{k,0}(\C{4})\right) \\ 
&= (k+1)^4 - 8k(k^2+1) = (k-1)^4.\end{align*}
We can define the volumetric degrees of freedom for the 0-form proxy $u$ as follows
\begin{equation}\label{eq:tesseract0a}
	\Sigma_{vol}^{k,0}(\C{4}):=\left\{ u \rightarrow \int_{\C{4}} u q, \qquad q \in Q^{k-2,k-2,k-2,k-2}(\C{4})\right\}.    
\end{equation} 
It then remains for us to prove unisolvency.

\begin{lemma}
    Let $u\in \Vk{k}{0}(\C{4})$ be a polynomial 0-form for which all the degrees of freedom $\Sk{k}{0}(\C{4})$ vanish. Then $u \equiv 0$. \label{tesseract_lemma_0}
\end{lemma}

\begin{proof}
    We first note that 
    \begin{align*}
       \text{dim}\left(\Sigma^{k,0}(\C{4})\right) =  
       \text{dim}\left(\Vk{k}{0}(\C{4})\right) =
       \text{dim}\left(\Sigma_{trace}^{k,0}(\C{4})\right) + \text{dim}\left(\Sigma_{vol}^{k,0}(\C{4})\right).
    \end{align*}
    Unisolvency will follow immediately if we can show that the vanishing of all degrees of freedom for $u\in \Vk{k}{0}(\C{4})$ implies the vanishing of $u$.
	
    Since all the trace degrees of freedom of the form given by Eqs.~\eqref{eq:edge-0}, \eqref{eq:square-0}, and \eqref{eq:hex0} vanish, the polynomial 0-form $u$ has vanishing traces, and is hence in $\Vkcirc{k}{0}(\C{4})$. It therefore has the form in Eq.~\eqref{eq:0formCbubble}, i.e.,
    \begin{align*}
        u =  \left(\Pi_{n=1}^4 (1-x_n^2) \right) \widehat{u},
    \end{align*}
    where
    \begin{align*}
        \widehat{u} = \sum_{ij\ell m} u_{ij \ell m} p_i(x_1)p_j(x_2)p_{\ell}(x_3)p_{m}(x_4),
    \end{align*}
    and where
    \begin{align*}
        p_i, p_j, p_{\ell}, p_{m} \in P^{k-2}(\C{1}).
    \end{align*}
    The proof follows immediately by choosing the test function $q = \widehat{u}$, and thereafter substituting this function and $u$ (from above) into Eq.~\eqref{eq:tesseract0a}. The vanishing of the associated volumetric degrees of freedom is only possible if $\widehat{u}$ vanishes.
\end{proof}

\subsection{Dofs for 1-forms on $\C{4}$}
We recall that 
\begin{align*}
\Vk{k}{1}(\C{4}) := Q^{k-1,k,k,k}\times Q^{k,k-1,k,k}\times Q^{k,k,k-1,k}\times Q^{k,k,k,k-1},
\end{align*}
and hence
\begin{equation*}
	  \text{dim}\left(\Vk{k}{1}(\C{4})\right) = \text{dim}\left(\Sigma^{k,1}(\C{4})\right) = 4k(k+1)^3.
\end{equation*} 
As stated previously, we will use well-known 1-form degrees of freedom on the edges, faces, and facets.
In particular, we can use the degrees of freedom associated with these objects (defined in Eqs.~\eqref{eq:edge-1}, \eqref{eq:square-1}, and \eqref{eq:hex1}) to obtain
\begin{align*}
	\text{dim}\left(\Sigma_{trace}^{k,1}(\C{4})\right)&:= 32 \, \text{dim}\left(P^{k-1}(\C{1}) \right) + 24 \, \text{dim} \left(Q^{k-2,k-1} \times Q^{k-1,k-2}(\C{2})\right) \\
    &+ 8 \, \text{dim}\left( Q^{k-1,k-2,k-2} \times Q^{k-2,k-1,k-2} \times Q^{k-2,k-2,k-1}(\C{3})\right) \\
	&= 32k+48k(k-1) + 24 k(k-1)^2 = 8k(3k^2+1).
\end{align*}
Next, we can specify the number of volumetric degrees of freedom as follows
\begin{align*}
\text{dim}\left(\Sigma_{vol}^{k,1}(\C{4})\right) &= \text{dim}\left(\Sigma^{k,1}(\C{4})\right) - \text{dim}\left(\Sigma_{trace}^{k,1}(\C{4})\right) \\
&= 4k(k+1)^3 - 8k(3k^2+1) = 4k(k-1)^3. 
\end{align*}
In addition, we can explicitly specify the volumetric degrees of freedom for a 1-form proxy $E$ as follows
\begin{align}
	\nonumber &\Sigma_{vol}^{k,1}(\C{4}):=\Bigg\{ E \rightarrow \int_{\C{4}} E \cdot q, \\
  &q \in Q^{k-1,k-2,k-2,k-2}\times Q^{k-2,k-1,k-2,k-2}\times Q^{k-2,k-2,k-1,k-2}\times Q^{k-2,k-2,k-2,k-1} \Bigg\}. \label{eq:tesseract1a} 
\end{align}
It then remains for us to prove unisolvency.

\begin{lemma}
	Let $E\in \Vk{k}{1}(\C{4})$ be a polynomial 1-form for which all the degrees of freedom $\Sk{k}{1}(\C{4})$ vanish. Then $E \equiv 0.$ \label{tesseract_lemma_1}
\end{lemma}

\begin{proof}
    We note that 
    \begin{align*}
       \text{dim}\left(\Sigma^{k,1}(\C{4})\right) =  
       \text{dim}\left(\Vk{k}{1}(\C{4})\right) =
       \text{dim}\left(\Sigma_{trace}^{k,1}(\C{4})\right) + \text{dim}\left(\Sigma_{vol}^{k,1}(\C{4})\right).
    \end{align*}
    Unisolvency will follow immediately if we can show that the vanishing of all degrees of freedom for $E\in \Vk{k}{1}(\C{4})$ implies the vanishing of $E$.
	
	Since all the trace degrees of freedom of the form given by Eqs.~\eqref{eq:edge-1}, \eqref{eq:square-1}, and \eqref{eq:hex1} vanish, the polynomial 1-form $E$ has zero traces, and is hence in $\Vkcirc{k}{1}(\C{4})$. It therefore has the form in Eq.~\eqref{eq:1formCbubble}, i.e.,
    \begin{align*}
        E =  \begin{bmatrix}
			\left(\Pi_{n=1, n\neq 1}^{4}(1- x_{n}^2)\right) \widehat{E}^{1} \\[1.0ex]
                \left(\Pi_{n=1, n\neq 2}^{4}(1- x_{n}^2)\right) \widehat{E}^{2}\\[1.0ex]
                \left(\Pi_{n=1, n\neq 3}^{4}(1- x_{n}^2)\right) \widehat{E}^{3}\\[1.0ex]
			\left(\Pi_{n=1, n\neq 4}^{4}(1- x_{n}^2)\right) \widehat{E}^{4} 
		\end{bmatrix},
     \end{align*}
    where
    \begin{align*}
     \widehat{E}:= \begin{bmatrix} \widehat{E}^{1} \\[1.0ex] \widehat{E}^{2} \\[1.0ex] \widehat{E}^{3} \\[1.0ex] \widehat{E}^{4} \end{bmatrix}  = \begin{bmatrix}
			\sum_{ij \ell m} E_{ij \ell m}^{1} p_i(x_1)p_j(x_2)p_{\ell}(x_3)p_{m}(x_4)\\[1.0ex]
                \sum_{ij \ell m} E_{ij \ell m}^{2} p_i(x_2)p_j(x_1)p_{\ell}(x_3)p_{m}(x_4)\\[1.0ex]
                \sum_{ij \ell m} E_{ij \ell m}^{3} p_i(x_3)p_j(x_1)p_{\ell}(x_2)p_{m}(x_4)\\[1.0ex]
			\sum_{ij \ell m} E_{ij \ell m}^{4} p_i(x_4)p_j(x_1)p_{\ell}(x_2)p_{m}(x_3) 
		\end{bmatrix}, \qquad  p_i\in P^{k-1}(\C{1}), \quad p_{j},p_{\ell},p_{m} \in P^{k-2}(\C{1}).
     \end{align*}
     %
     %and where
     %
     %\begin{align*}
      %  p_i\in P^{k-1}(\C{1}), \qquad p_{j},p_{\ell},p_{m} \in P^{k-2}(\C{1}).
    %\end{align*}
    %
    The proof follows immediately by choosing the test function $q = \widehat{E}$, and thereafter substituting this function and $E$ (from above) into Eq.~\eqref{eq:tesseract1a}. It is clear that the vanishing of the volumetric degrees of freedom is only possible if $\widehat{E}$ vanishes.
\end{proof}

\subsection{Dofs for 2-forms on $\C{4}$}
By now our methodology is clear. We can enumerate the dimensions of $\Vk{k}{2}(\C{4})$ as follows
\begin{align*}
	\text{dim}\left(  \Vk{k}{2}(\C{4})\right) & = \text{dim}\left(\Sigma^{k,2}(\C{4})\right) = 6 \, \text{dim} \left(Q^{k,k,k-1,k-1}\right) = 6 k^2(k+1)^2.
\end{align*}
The face and facet degrees of freedom are well-defined for 2-forms. If we use the degrees of freedom defined in Eqs.~\eqref{eq:square-2} and \eqref{eq:hex2} as our trace degrees of freedom, we obtain
\begin{align*}
	\text{dim}\left(\Sigma_{trace}^{k,2}(\C{4})\right) &= 24 \, \text{dim}\left(Q^{k-1,k-1}(\C{2})\right) + 24 \, \text{dim} \left(Q^{k-2,k-1,k-1}(\C{3})\right) \\
    &= 24 k^2 + 24 k^2(k-1) =24 k^3. 
\end{align*}
We can now define the set of volumetric degrees of freedom for the 2-form proxy $F$ as follows
\begin{eqnarray}
	\Sigma_{vol}^{k,2}(\C{4}):=\left\{ F \rightarrow \int_{\C{4}
	} F:q,  \qquad q    \in \VtoM{\begin{bmatrix}
			Q^{k-1,k-1,k-2,k-2}\\[1.0ex]
			Q^{k-1,k-2,k-1,k-2}\\[1.0ex]
			Q^{k-1,k-2,k-2,k-1}\\[1.0ex]
			Q^{k-2,k-1,k-1,k-2}\\[1.0ex]
			Q^{k-2,k-1,k-2,k-1}\\[1.0ex]
			Q^{k-2,k-2,k-1,k-1}
			 \end{bmatrix}  }
	\right\}. \label{eq:tesseract2a}
\end{eqnarray}
Hence, we have that
\begin{align*}
\text{dim}\left(\Sigma_{vol}^{k,2}(\C{4})\right) &= \text{dim}\left(\Sigma^{k,2}(\C{4})\right) - \text{dim}\left(\Sigma_{trace}^{k,2}(\C{4})\right) \\
&= 6k^2(k+1)^2 - 24k^3 = 6k^2(k-1)^2. 
\end{align*}
It then remains for us to prove unisolvency.

\begin{lemma}
	Let $F\in \Vk{k}{2}(\C{4})$ be a polynomial 2-form for which all the degrees of freedom $\Sk{k}{2}(\C{4})$ vanish. Then $F \equiv 0.$ \label{tesseract_lemma_2}
\end{lemma}

\begin{proof}
	As before, we note
	\begin{align*}
	    \text{dim}\left(\Sk{k}{2} (\C{4})\right) = \text{dim}\left(\Vk{k}{2}(\C{4})\right) = \text{dim} \left(\Sigma_{trace}^{k,2} (\C{4}) \right) + \text{dim}\left(\Sigma_{vol}^{k,2}(\C{4})\right).
	\end{align*}
	Unisolvency will follow immediately if we can show that the vanishing of all degrees of freedom for $F\in \Vk{k}{2}(\C{4})$ implies the vanishing of $F$.
 
    Since all the trace degrees of freedom of the form given by Eqs.~\eqref{eq:square-2} and \eqref{eq:hex2} vanish, the polynomial 2-form $F$ has zero traces, and is hence in $\Vkcirc{k}{2}(\C{4})$. It therefore has the form given by Eq.~\eqref{eq:2formCbubble}, i.e., 
    \begin{align*}
	  F&= \VtoM{\begin{bmatrix}
				\left(\Pi_{n=1, n\neq1,2}^{4} (1-x_{n}^{2}) \right)                     \widehat{F}^{1} \\[1.0ex]
                    \left(\Pi_{n=1, n\neq1,3}^{4} (1-x_{n}^{2}) \right) \widehat{F}^{2} \\[1.0ex]
                    \left(\Pi_{n=1, n\neq1,4}^{4} (1-x_{n}^{2}) \right) \widehat{F}^{3} \\[1.0ex]
                    \left(\Pi_{n=1, n\neq2,3}^{4} (1-x_{n}^{2}) \right) \widehat{F}^{4} \\[1.0ex]
                    \left(\Pi_{n=1, n\neq2,4}^{4} (1-x_{n}^{2}) \right) \widehat{F}^{5} \\[1.0ex]
                    \left(\Pi_{n=1, n\neq3,4}^{4} (1-x_{n}^{2}) \right) \widehat{F}^{6}
		\end{bmatrix}  }, 
     \end{align*}
     where 
     \begin{align*}
         \widehat{F}:= \begin{bmatrix} \widehat{F}^{1} \\[1.0ex] \widehat{F}^{2} \\[1.0ex] \widehat{F}^{3} \\[1.0ex] \widehat{F}^{4} \\[1.0ex] \widehat{F}^{5} \\[1.0ex] \widehat{F}^{6} \end{bmatrix} = \begin{bmatrix}
				\sum_{ij \ell m} F_{ij \ell m}^{1}                     p_i(x_1)p_j(x_2)p_{\ell}(x_3)p_{m}(x_4) \\[1.0ex]
                    \sum_{ij \ell m} F_{ij \ell m}^{2} p_i(x_1)p_j(x_3)p_{\ell}(x_2)p_{m}(x_4) \\[1.0ex]
                    \sum_{ij \ell m} F_{ij \ell m}^{3} p_i(x_1)p_j(x_4)p_{\ell}(x_2)p_{m}(x_3) \\[1.0ex]
                    \sum_{ij \ell m} F_{ij \ell m}^{4} p_i(x_2)p_j(x_3)p_{\ell}(x_1)p_{m}(x_4) \\[1.0ex]
                    \sum_{ij \ell m} F_{ij \ell m}^{5} p_i(x_2)p_j(x_4)p_{\ell}(x_1)p_{m}(x_3) \\[1.0ex]
                    \sum_{ij \ell m} F_{ij \ell m}^{6} p_i(x_3)p_j(x_4)p_{\ell}(x_1)p_{m}(x_2)
		\end{bmatrix}, \qquad   p_{i}, p_{j} \in P^{k-1}(\C{1}), \qquad p_{\ell}, p_{m} \in P^{k-2}(\C{1}).
     \end{align*}
    The proof follows immediately by choosing the test function $q = \VtoM{ \widehat{F}}$, and thereafter substituting this function and $F$ (from above) into Eq.~\eqref{eq:tesseract2a}. The vanishing of the volumetric degrees of freedom is only possible if $\widehat{F}$ vanishes.
\end{proof}

\subsection{Dofs for 3-forms on $\C{4}$}

We start by enumerating the number of dimensions
\begin{align*}
\text{dim} \left(\Vk{k}{3}(\C{4})\right) = \text{dim}\left(\Sigma^{k,3}(\C{4}) \right) = 4k^3(k+1).  
\end{align*}
%
% We'll define the degrees of freedom in terms of trace and volume degrees, 
% %
% \begin{align*}
% 	\Sk{k}{3}(\C{4}) = \Sk{k}{3}_{trace} (\C{4}) \cup \Sk{k}{3}_{vol} (\C{4}).
% \end{align*}
%
There are 8 hexahedral facets for $\C{4}$, and therefore, upon using the associated hexahedral degrees of freedom (Eq.~\eqref{eq:hex3}) to define $\Sk{k}{3}_{trace}(\C{4})$, we obtain
\begin{align*}
	\text{dim}\left( \Sk{k}{3}_{trace}(\C{4})\right) = 8 k^3.
\end{align*}
Next, we need to specify degrees of freedom for the volume
\begin{align*}
     \text{dim}\left(\Sk{k}{3}_{vol}(\C{4})\right) &= \text{dim}\left(\Sk{k}{3}(\C{4})\right) - \text{dim}\left(\Sk{k}{3}_{trace}(\C{4})\right) = 4k^3(k-1).
\end{align*}
	We propose the following volumetric degrees of freedom for the 3-form proxy $G$
    \begin{align}
		\nonumber &\Sk{k}{3}_{vol}(\C{4}):=\Bigg\{ G \rightarrow \int_{\C{4}} G \cdot q,  \\
		&q \in Q^{k-2,k-1,k-1,k-1}\times Q^{k-1,k-2,k-1,k-1}\times Q^{k-1,k-1,k-2,k-1}\times Q^{k-1,k-1,k-1,k-2} \Bigg\}. \label{eq:tesseract3a} 
    \end{align}
    It then remains for us to prove unisolvency.

    \begin{lemma}
        Let $G\in \Vk{k}{3}(\C{4})$ be a polynomial 3-form for which all the degrees of freedom $\Sk{k}{3}(\C{4})$ vanish. Then $G \equiv 0.$
    \end{lemma}

\begin{proof}
    The proof is straightforward, as it directly follows the proofs of Lemmas~\ref{tesseract_lemma_0}, \ref{tesseract_lemma_1}, and \ref{tesseract_lemma_2} with $G$ and $q$ constructed using the definition of the bubble space, $\Vkcirc{k}{3}(\C{4})$, in Eq.~\eqref{eq:3formCbubble}.
\end{proof}

\subsection{Dofs for 4-forms on $\C{4}$}
	
	Trace degrees of freedom are not defined for 4-forms. Therefore, the degrees of freedom for $q \in \Vk{k}{4}(\C{4})$ are simply
	\begin{equation}
		\Sk{k}{4}(\C{4}):=\left\{ q \rightarrow \int_{\C{4}} qp, \qquad p \in Q^{k-1,k-1,k-1,k-1} \right\}. \label{eq:tesseract4a}
	\end{equation}
	Hence
	\begin{align*}
	    \text{dim}\left(\Sk{k}{4}(\C{4}) \right) = k^4.
	\end{align*}
	It then remains for us to prove unisolvency.

\begin{lemma}
    Let $q\in \Vk{k}{4}(\C{4})$ be a polynomial 4-form for which all the degrees of freedom $\Sk{k}{4}(\C{4})$ vanish. Then $q \equiv 0.$
\end{lemma}

\begin{proof}
    We begin by choosing a generic $q$, such that
    \begin{align*}
        q = \sum_{ij\ell m} q_{ij \ell m} p_i(x_1)p_j(x_2)p_{\ell}(x_3)p_{m}(x_4),
    \end{align*}
    where
    \begin{align*}
        p_{i}, p_{j}, p_{\ell}, p_{m} \in P^{k-1}(\C{1}).
    \end{align*}
    The proof follows immediately by setting the test function $p = q$, and thereafter substituting this function and $q$ (from above) into Eq.~\eqref{eq:tesseract4a}. Under these circumstances, the degrees of freedom are only guaranteed to vanish if the polynomial $q$ vanishes.
\end{proof}